\documentclass[12pt,reqno] {amsproc}
\usepackage{graphicx}
\usepackage{tikz}
\usepackage{latexsym, enumitem, fancyhdr, bm}
\usepackage[utf8]{inputenc}
\usepackage{amsfonts}
\usepackage{mathrsfs}
\usepackage{xcolor}
\usepackage{todonotes}
\usepackage{ulem}
\usepackage{hyperref}
\usepackage{mathtools}
\usepackage{amsmath}
\usepackage{amssymb}
\usepackage{amsthm}
\newtheorem{proposition}{Proposition}
\newtheorem{lemma}{Lemma}
\newtheorem{theorem}{Theorem}

\newtheorem{corollary}{Corollary}
\theoremstyle{remark}
\newtheorem{remark}[theorem]{Remark}
\theoremstyle{remark}
\theoremstyle{definition}
\newtheorem{definition}[theorem]{Definition}
\newtheorem{example}[theorem]{Example}

\title[The Cauchy dual subnormality problem in two variables]{The Cauchy dual subnormality in two variables}
\author[S. Dey]{Soumyadip Dey}
\address[S. Dey]{School of Mathematics and Computer Science\\Indian Institute of Technology Goa, India}
 \email{soumyadipdey03@gmail.com\\ soumyadip22232101@iitgoa.ac.in}
\begin{document}
\keywords{Cauchy Dual, Subnormality, toral isometry, Weighted Shift, Dirichlet-type Spaces, Left-inverse commuting tuple}
\subjclass{Primary 47B20, 47B37; Secondary 44A60,47A13}
\begin{abstract}
The Cauchy dual subnormality problem asks whether the Cauchy dual of a $2$-isometry is subnormal. This problem has attracted considerable attention due to its connections with moment theory, operator theory, and complex function theory. In particular, it has been closely related to operator-theoretic models on function spaces such as de Branges--Rovnyak spaces and Dirichlet-type spaces. In this paper, we investigate the Cauchy dual subnormality problem in the multivariable setting of toral $2$-isometries.  We show that if $\mathbf{T} = (T_1, T_2)$ is an analytic left-inverse commuting pair of toral $2$-isometries satisfying the joint kernel condition, then it is unitarily equivalent to an operator-valued weighted shift with invertible weights $\{W_{I}^{(j)}:j=1,2,I\in \mathbb{Z}_{+}^2\},$ where the initial weights $W_{0,0}^{(1)}$ and $W_{0,0}^{(2)}$ are positive operators. Moreover, if these initial weights commute, then the Cauchy dual $\mathbf{T}' := (T_1', T_2')$ is  subnormal. We also construct an example in which the initial weights do not commute, and the corresponding Cauchy dual tuple $\mathbf{T'}=(T_1',T_2')$ fails to be subnormal. Finally, we establish a complete set of unitary invariants for completely non- unitary toral $2$-isometries.
\end{abstract}
\maketitle
\section{Introduction and the main result}
The Cauchy dual subnormality problem (CDSP) for $m$-isometric $n$-tuples concerns the joint subnormality of their Cauchy duals. In view of some recent developments in operator theory on function spaces (see \cite{AM},\cite{LS}), one may recast the Cauchy dual subnormality problem (CDSP) as the problem of subnormality of the Cauchy dual \(M_z'\) of the multiplication operator \(M_z\) acting on a de Branges--Rovnyak space \(\mathcal{H}(B)\), where \(B\) is a vector-valued rational function. Although the CDSP has been extensively studied for one-variable $2$-isometries, yielding deep connections with moment theory and function-theoretic operator models (see, e.g., \cite{A.A,B.C,CH}), its multivariable counterpart remains largely unexplored. In this paper, we investigate the CDSP for left-inverse commuting toral $2$-isometries satisfying the joint kernel condition. The recent work of Chavan et al established a function-theoretic approach to the Cauchy Dual subnormality Problem for cyclic $2$-isometries through de Branges--Rovnyak spaces (see \cite{CH}). Another early result towards the solution of the Cauchy dual subnormality problem (CDSP) asserts that the Cauchy dual of every concave operator is a hyponormal contraction (see \cite[Equation (26)]{SH}). This result was later generalized in \cite[Theorem 3.1]{S.C}, where it was shown that the Cauchy dual of every concave operator is power hyponormal. Around the same period, the Cauchy Dual Subnormality Problem (CDSP) was resolved affirmatively for several important classes of $2$-isometries. In particular, a positive solution was obtained for $\Delta_T$-regular $2$-isometries in \cite[Theorem 3.4]{B.C} and for $2$-isometric operator-valued weighted shifts in \cite[Theorems 2.5 and 3.3]{A.A}; see also \cite[Corollary 6.2]{CJ} for another class of $2$-isometries for which the problem admits an affirmative answer. On the other hand, \cite[Examples 6.6 and 7.10]{A.A} demonstrated that the Cauchy dual of a $2$-isometric weighted shift on a directed tree, including adjacency operators, need not be subnormal, thereby showing that the CDSP does not hold in full generality.

 Let $\mathcal{H}$ be a complex Hilbert space. Denote by $B(\mathcal{H})$ the \(C^{*}\)-algebra of all bounded linear operators on $\mathcal{H}$.
 An operator $T$ is said to be {\it subnormal} if there exists a Hilbert space $\mathcal{K}$ containing $\mathcal{H}$ and a normal operator $N\in B(\mathcal{K})$ such that $Nh=Th$ for every $h\in\mathcal{H}$. In what follows, the symbol $T$ shall always denote a bounded linear operator on the Hilbert space $\mathcal H,$ unless stated otherwise. The study of subnormal operators is closely connected with moment problems and has been extensively developed over the past several decades. In particular, the pioneering work of Curto and Putinar established deep connections between subnormality, polynomial hyponormality, and moment theory. Their results on the existence of non-subnormal polynomially hyponormal operators \cite{R.E} and on nearly subnormal operators and moment problems \cite{R.P} have had a significant impact on the development of modern subnormal operator theory. Given a positive integer $m$, we say that an operator $T$ is an \textit{$m$-isometry (or $m$-isometric)} if $B_m(T)=0$, where
\[
B_m(T):=\sum_{k=0}^{m} (-1)^k \binom{m}{k} \, T^{*k} T^k.
\]
Clearly, a $1$-isometry is an isometry. We say that $T$ is $2$-hyperexpansive (resp., completely hyperexpansive) if $B_2(T) \leq 0$ (resp., $B_m(T) \leq 0$ for all positive integers $m$). The notion of an $m$-isometric operator has been invented by Agler see \cite[p.11]{A.J}.
An operator $T$ is said to be left-invertible if $T^*T$ is invertible. In this case, the canonical left inverse of $T$ is given by $L_T := (T^*T)^{-1} T^*.$ The \textit{Cauchy dual} operator $T'$ of a left-invertible operator $T$ is define by $T'=T(T^*T)^{-1}$. 
 If $T$ is left-invertible, then so is its Cauchy dual $T'$, and the following identities hold: $(T')' = T$ and $(T')^* T' = (T^* T)^{-1}.$ The notion of the Cauchy dual operator was introduced and studied by Shimorin in the context of the wandering subspace problem for Bergman-type operators \cite{S.S}. 
 This technique has been further employed in \cite{S.C} to establish Berger--Shaw type theorems for $2$-hyperexpansive (or, in Shimorin's terminology, concave) operators. 
 It is also important to note that if a bounded linear operator $T$ on $\mathcal H$ is a $2$-isometry, then $T$ is left-invertible and $T'$ is a contraction. Indeed by \cite[Lemma 1]{S.R}, we have $\|T f\| \geq \|f\|$ for all $f \in \mathcal{H}$, which implies that  $T$ is left-invertible. Moreover, writing the polar decomposition $T = U |T|$, where $|T|$ denotes the positive part of $T,$ one can obtain
$\|T'\| = \|U |T|^{-1}\| \leq 1.$
An operator \(T\) is said to be {\it hyponormal} if the self-commutator $[T^*,T] := T^*T - TT^*$ of \(T\) is a positive operator. Recall that \(T\) is hyponormal if and only if there exists a contraction \(C\) in \(B(\mathcal{H})\) such that $T^* = CT$ (see \cite{M.M}). There is a close connection between the Cauchy dual of a \(2\)-hyperexpansive operator and hyponormality (see \cite[Theorem 2.9]{S.C}). In particular, the Cauchy dual of a left-invertible \(2\)-hyperexpansive operator is a hyponormal contraction. As a consequence, the self-commutators of \(T\) and its Cauchy dual \(T'\) are intimately related. Indeed, an observation made in the proof of \cite[Theorem 2.9]{S.C} shows that for every left-invertible operator \(T\) in \(B(\mathcal{H})\), if $T^{*}[T^{*},T]T \geq 0$ then $[T'^{*},T'] \geq 0.$ 
This implication, recorded in \cite[Remark 2.10]{S.C}, provides a useful criterion for establishing the hyponormality of \(T\) through the positivity of its self-commutator.

Recently, Anand et al posed the question of whether the Cauchy dual of a $2$-isometry is a subnormal contraction. In 2018, they proved that, upon imposing an additional kernel condition $T^*T( \ker T^*)\subseteq \ker T^*,$ the Cauchy dual is subnormal see \cite[Theorem 3.3]{A.A}.
A commuting $d$-tuple of operators $\mathbf{T} = (T_1, \ldots, T_d)$ on $\mathcal{H}$ is said to be \textit{jointly subnormal} if there exist a Hilbert space $\mathcal{K}$ containing $\mathcal{H}$ and a commuting $d$-tuple $\mathbf{N} = (N_1, \ldots, N_d)$ of normal operators on $\mathcal{K}$ such that $\mathcal{H}$ is invariant under $N_j$ and $T_j = N_j|_{\mathcal{H}},$ for each $j = 1, \ldots, d$. 
The \textit{Cauchy dual} of a $d$-tuple $\mathbf{T}=(T_1, \ldots, T_d)$ of left-invertible operators is defined as $\mathbf{T}' := (T_1', \ldots, T_d')$. 
It is also important to note that if $\mathbf{T} = (T_1, \ldots, T_d)$ is a tuple of $2$-isometries, then each $T_j$ is left-invertible and each $T_j'$ is a contraction by \cite[Lemma 1]{S.R}.
Given a positive integer $k$, a commuting $d$-tuple $\mathbf{T} = (T_1, \ldots, T_d)$ of bounded linear operators is called a {\it toral $k$-isometry} if $B_{\beta}(\mathbf{T}) = 0$ for all $\beta\in\mathbb Z_+^d$ with $|\beta|=k$,
where
\[
B_{\beta}(\mathbf{T}) :=
\sum_{\alpha \in \mathbb{Z}_+^d,\; 0 \leq \alpha \leq \beta}
(-1)^{|\alpha|}
\binom{\beta}{\alpha}
\mathbf{T}^{*\alpha} \mathbf{T}^\alpha.
\]
Here, $\mathbf{T}^{\alpha} := T_1^{\alpha_1}\cdots T_d^{\alpha_d}$ and $\alpha \leq \beta$ if $\alpha_j \leq \beta_j$ for every $j = 1, \ldots, d$. 

The Cauchy dual subnormality problem has been studied for various classes of $m$-isometries and multivariable weighted shifts. In particular, \cite{Nailwal2024} solved Cauchy dual subnormality problem for torally expansive toral $3$-isometric weighted $2$-shifts using complete monotonicity and Hausdorff moment problems. Motivated by this, we study the problem for operator-valued weighted $2$-variable shifts associated with toral $2$-isometries. 

As results in this article are best formulated for toral $2$-isometries, we record the definition of toral $2$-isometry separately. See also \cite{V.A}.

\begin{definition}\label{Toral $2$ isometry}
Let $\mathbf{T} = (T_1, \ldots, T_d)$ be a commuting $d$-tuple of operators on $B(\mathcal{H})$.  
We say that $\mathbf{T}$ is said to be a \textit{toral $2$-isometry} if $I - T_i^* T_i - T_j^* T_j + T_j^* T_i^* T_i T_j = 0$ for each $i, j = 1, \ldots, d.$  
\end{definition}
Also  important to note that, if $\mathbf T$ is toral $2$-isometry then for any integers $m, n \geq 0,$ one gets the identity 
\begin{eqnarray}\label{Power toral $2$-isometry condition}
{T_1^{*}}^m{T_2^{*}}^nT_2^nT_1^m=mT_1^*T_1 +nT_2^*T_2 -(m+n-1)I,
\end{eqnarray}
(see \cite[Lemma 6.1]{S.B}).
It was recently proved that if $\mathbf{T} = (T_1, T_2)$ is a commuting pair of toral $2$-isometric scalar-valued weighted shifts with non-zero weights, then the Cauchy dual tuple  $\mathbf{T'}=(T_1',T_2')$ is subnormal (see \cite[Theorem 4.9]{A.K}).

 A tuple $(T_1,\ldots ,T_d)$ of  left invertible  commuting  operators is said to be a {\it left-inverse commuting} tuple if $L_iT_j=T_jL_i$ for $1\leq i\neq j\leq d,$ where $L_j:=(T_j^*T_j)^{-1}T_j^*$, (see \cite[Definition 1]{MJ}). A commuting tuple $\mathbf{T}=(T_1,\ldots ,T_d)$ on $\mathcal{H}$ is {\it analytic} if for each $1\leq j\leq d$ we have $ \bigcap_{n\in \mathbb{Z}_{+}}{T_{j}}^{n}(\mathcal{H})=\{0\}.$

The notion of a left-inverse commuting tuple plays a crucial role in establishing a model class for analytic toral $2$-isometric $2$-tuples of operators on $\mathcal{E}$-valued Dirichlet-type spaces over $\mathbb{D}^2$ (see \cite[Theorem 3.3]{Mul}).
 Motivated by the results of \cite[Theorem 3.3]{A.A} and \cite{A.K}, we investigate an analogous problem for left-inverse commuting tuples of toral $2$-isometries. We now proceed to formulate the following version of this  problem.
 
\textbf{Problem.} Is the Cauchy dual of every left-inverse commuting toral $2$-isometry  subnormal?
\\
In this paper, we solve this problem by imposing an additional joint kernel condition. We say that an operator tuple $\mathbf{T}=(T_1,T_2)$ on $B(\mathcal{H})$ satisfies \textit{the joint kernel condition} if
\begin{eqnarray}\label{Kernel condition}
T_i^* T_i (\ker \mathbf{T^*}) \subseteq (\ker \mathbf{T^*}), \quad \text{for } i=1,2. 
\end{eqnarray}
where for any pair $\mathbf S=(S_1,S_2)$ of bounded linear operators, $\ker \mathbf{S^*}:=\ker S_1^* \cap \ker S_2^*.$ By the square root lemma (see \cite[Theorem 2.4.4]{Simon2015}), \eqref{Kernel condition} holds if and only if $|T_i|(\ker \mathbf{T^*}) \subseteq \ker \mathbf{T^*}$ for $i=1,2,$ 
where $|T_i|$ denotes the positive part of $T_i$ for $i=1,2$.

The following is the main theorem of this article.
\begin{theorem}[Main Theorem]\label{Main theorem}
Let $\mathbf{T} = (T_1, T_2)$ be a pair of left-inverse commuting  analytic toral $2$-isometries on $B(\mathcal{H})$ satisfying the joint kernel condition. Assume that 
$|T_1|\,|T_2| = |T_2|\,|T_1|$  \text{on }$\ker \mathbf{T}^*.$ 
Then the Cauchy dual pair $\mathbf{T}'=(T'_1,
T'_2)$ consists of subnormal contractions such that $T_i'^* T_i' \big( \ker \mathbf{T}'^* \big) \subseteq \ker \mathbf{T}'^*$ for $i=1,2 .$
\end{theorem}
The remainder of the paper is organized as follows.
In Section \ref{Model Section}, we provide a characterization of left-inverse commuting analytic  toral $2$-isometries satisfying the joint kernel condition (see Theorem~\ref{Orthogonality}). In Subsection \ref{Subsection of operator-valued multishift}, we study operator-valued multishifts, and we construct a model for such tuples of operators using operator-valued $2$-shifts (see Theorem~\ref{Representation of weights}). We also include a brief discussion of cyclic analytic toral $2$-isometries satisfying the joint kernel condition (see Proposition~\ref{Example}) and also discuss the connection between the subnormality of operator-valued multishifts and Stieltjes moment sequences (see Theorem \ref{Subnormality criteria}). 
In Section \ref{Section on Cauchy Dual Subnormality}, we provide  an integral representation for the toral $2$-isometric operator-valued weighted shift with in the class of $\ell^2_{IWP}(\mathcal{H})$, when initial two operator weights are commuting (see Proposition \ref{Integral representation}). 
Finally we prove that the Cauchy dual $\mathbf{T'} = (T_1', T_2')$ of a left-inverse commuting analytic toral $2$-isometry $\mathbf{T} = (T_1, T_2)$ satisfying the joint kernel condition, and such that $|T_1||T_2| = |T_2||T_1|$ on $\ker \mathbf{T}^*$, is a  subnormal contraction. We also show that if the analyticity assumption is removed and the joint kernel condition is strengthened by separate kernel conditions, then the Cauchy dual tuple $\mathbf{T'}=(T_1',T_2')$ is subnormal (see Corollary~\ref{Seperate kernel condition}). 
We provide an example where $|T_1||T_2| \neq |T_2||T_1|$ on $\ker \mathbf{T}^*$ and the Cauchy dual tuple  $\mathbf{T'}=(T_1',T_2')$ fails to be  subnormal (see Example \ref{Example 3}). 
At the end (see Section \ref{Characterizations of Pure Toral $2$-Isometries}) we answer the question when two analytic left-invertible commuting toral $2$-isometries satisfying the joint kernel condition are unitarily equivalent (see Theorem \ref{thm:unitary-equivalence} and Theorem \ref{thm:UE-commuting}).

Now we fix notation and terminology. Let $\mathbb{Z}$, $\mathbb{R},$ and $\mathbb{C}$ stand for the sets of integers, real numbers, and complex numbers respectively. Denote by $\mathbb{N}$, $\mathbb{Z}_{+},$ and $\mathbb{R}_{+}$ the sets of positive integers, non-negative integers, and non-negative real numbers respectively. Given a topological space $X$, let $\mathfrak{B}(X)$ denote the $\sigma$-algebra of all Borel subsets of $X.$ For $a\in \mathbb{R}$,  $\delta_{a}$ stands for the Borel probability measure on $\mathbb{R}$ supported on $\{a\}$. 
For a positive integer $d$,  let $\mathbb{Z}_{+}^d$ stand for the $d$-fold Cartesian product of $\mathbb{Z}_{+}$ with itself. 
Let $\alpha = (\alpha_1, \ldots, \alpha_d)$ and $\beta = (\beta_1, \ldots, \beta_d)$ be in $\mathbb{Z}_{+}^d$. 
Let $|\alpha|$ denote the sum $\alpha_1 + \cdots + \alpha_d$. 
We write $\alpha \leq \beta$ if $\alpha_j \leq \beta_j$ for every $j = 1, \ldots, d$. 
For $\alpha \leq \beta$, we define $\binom{\beta}{\alpha} := \prod_{j=1}^d \binom{\beta_j}{\alpha_j}$.
 The notation $B(\mathcal{H})$ will stand for the $C^*$-algebra of all bounded linear operators on a Hilbert space $\mathcal{H}$. The product semigroup $\mathbb{Z}_{+}^d\times \mathbb{Z}_{+}^d$ with involution $(a,b)^*=(b,a)$ for $a,b\in \mathbb{Z}_{+}^d,$ is a $*$-semigroup. Denote it by $\mathfrak{R}_{d}.$ Let $\mathbb{D}^2=\{(z_1,z_2):|z_1|<1,|z_2|<1 \}$ denoted by bidisc and  $\mathcal{D}(\mu_1,\mu_2)$ be a Dirichlet-type space on $\mathbb{D}^2,$ for some finite Borel measures $\mu_1,\mu_2$ on the unit circle $\mathbb{T}=\{z\in \mathbb{C}:|z|=1\}.$
\section{\textbf{Model for left-inverse commuting toral $2$-isometries with kernel condition}}\label{Model Section}
The goal of this section is to show that a pair of analytic left-inverse commuting toral $2$-isometry satisfying the joint kernel condition is unitarily equivalent to an operator-valued $2$-variable weighted shift (see Theorem \ref{Model}).
The following definition is taken from \cite[page 10]{R.J}.
\begin{definition}[D-slice ordering:]
    For each $k \in \mathbb{Z}$, define ${P}_k := \{(x,y) \in \mathbb{Z}^2 : x+y = k\}.$
Then the family $\{{P}_k\}_{k\in\mathbb{Z}}$ forms a collection of pairwise disjoint subsets of $\mathbb{Z}^2$ satisfying $\mathbb{Z}^2 = \bigsqcup_{k\in\mathbb{Z}} {P}_k .$
The \emph{D-slice ordering} on $\mathbb{Z}^2$ is defined as follows. For each $\ell\leq m$ in $\mathbb Z,$ let $(x_1,y_1)\in {P}_\ell$, and $(x_2,y_2)\in {P}_m .$
Then we declare $(x_1,y_1) < (x_2,y_2)$ if either
\begin{enumerate}
    \item $\ell < m$, or
    \item $\ell = m$ and $(x_1,y_1)$ precedes $(x_2,y_2)$ in the lexicographic ordering on ${P}_\ell \subseteq \mathbb{Z}^2$.
\end{enumerate}
\end{definition}
We begin by characterizing left-inverse commuting analytic toral $2$-isometric pairs satisfying the joint kernel condition, in analogy with the one-variable setting; (see \cite[Theorem 2.5]{A.A}). We shall also invoke \cite[Corollary 4.2]{MJ}, which states that, for a commuting tuple $\mathbf{T}=(T_1,\ldots,T_d)$ and $i\neq j$, the relation $L_iT_j=T_jL_i$ holds if and only if $\ker T_i^*$ is a $T_j$-reducing subspace of $\mathcal{H}$. This result will be used in the proof of the orthogonality property established in Theorem~\ref{Orthogonality}.

\begin{theorem}\label{Orthogonality}
Let $\boldsymbol T=(T_1, T_2)$ be a tuple of left-inverse commuting  toral $2$-isometries satisfying the joint kernel condition $T_j^* T_j (\ker \mathbf{T}^*) \subseteq \ker \mathbf{T}^*$, for $j=1,2$, then
\begin{eqnarray}\label{orthogonality_condition}
T_1^m T_2^n (\ker \mathbf{T}^*)
\perp
T_1^p T_2^q (\ker \mathbf{T}^*),
\quad \text{whenever } (m,n)\neq(p,q),
\end{eqnarray}
where $\ker \mathbf{T}^*=\ker T_1^*\cap\ker T_2^*$.
\end{theorem}
\begin{proof}
For each $N\in\mathbb Z_+,$ define ${P}_N:= \{(m,n) \in \mathbb{Z}_+^2 : m+n = N\}.$
The proof is based on induction on $N \in \mathbb{Z}_+$.  
As the inner product is conjugate linear, therefore it is enough to prove \eqref{orthogonality_condition} for $(m,n) \in {P}_N$ and $(p,q) \in {P}_{N+j}$ with $j \ge 0$. 

\textbf{Step 1:} $N=0.$

That is $(m,n)=(0,0)$ and $(p,q) \in {P}_j$ for $j\geq 1$. Since at least one of $p,q$ is non-zero, without loss of generality, we may assume that $p \neq 0.$ Then $\langle f, T_1^p T_2^q g \rangle
= \langle T_1^* f, T_1^{p-1} T_2^q g \rangle = 0,$ for all $f,g \in \ker \mathbf{T}^*$. 

\textbf{Step 2:}
Assume the condition \eqref{orthogonality_condition} holds for all $(m,n)\in  P_k$ and $(p,q)\in P_{k+j}$ for each $k\in \{0,\ldots,N\}$ for some $N\in\mathbb N$ and for all $j\in \mathbb Z_+$.  

\textbf{Step 3:}Let $N\in \mathbb{N.}$ Take $(m,n)\in {P}_{N+1}$ and $(p,q)\in {P}_{N+1+j}$, for some $j\geq 0.$ This step is divided into two primary cases which are further divided into several sub cases.
\begin{itemize}
\item[\textbf{Case 1:}] Suppose $j\geq 1$.\\
\textbf{1.} Let $m=0$ and $n=N+1$.
\begin{enumerate}[label=\alph*)]
\item If $p=0$ and $q\geq N+2$, then using $2$-isometry condition, we have
\begin{eqnarray*}
   \hspace{2cm} \langle T_2^{N+1}f,\, T_2^qg\rangle
&=&\big\langle (T_2^*)^{N+1}T_2^{N+1}f,\,
T_2^{q-(N+1)}g \big\rangle \\
&=&\big\langle
\big(2T_2^{*N}T_2^N-T_2^{*(N-1)}T_2^{N-1}\big)f,\,
T_2^{q-(N+1)}g
\big\rangle \\
&=&0.
\end{eqnarray*}
\item If $p\geq N+2$ and $q=0$, then 
$\langle T_2^{N+1}f,\, T_1^pg\rangle
=\langle T_2^Nf,\, T_2^*T_1^pg\rangle
=0,$ 
because $T_1^pg\in \ker T_2^*$.
\item Let $p\neq 0$ and $q\neq 0$. Then
$
\langle T_2^{N+1}f,\, T_1^pT_2^qg\rangle
=\langle T_1^*T_2^{N+1}f,\,
T_1^{p-1}T_2^qg\rangle
=0,$ since $T_2^{N+1}f\in \ker T_1^*$.
Similarly, for $m=N+1$ and $n=0$, we obtain
$\langle T_1^{N+1}f,\,
T_1^pT_2^qg\rangle =0$ for all $(p,q)\in {P}_{N+j}$, where $j\geq 2$.
\end{enumerate}
\textbf{2.} Let $m=1$ and $n=N$.
\begin{enumerate}[label=\alph*)]
\item If $p=0$ and $q\geq N+2$, then
$\langle T_1T_2^Nf,\, T_2^qg\rangle
=\langle T_2^Nf,\, T_1^*T_2^qg\rangle
=0,$ as $T_2g\in \ker T_1^*$.
\item If $p\geq N+2$ and $q=0$, then similarly, $\langle T_1T_2^Nf,\, T_1^pg\rangle =0.$
\item If $p\geq 1$ and $q\geq 1$, then using toral $2$-isometry, we have
\[
\begin{aligned}
\hspace{2cm} \langle T_1T_2^Nf,\, T_1^pT_2^qg\rangle
&=\left\langle
(T_2^*T_1^*T_1T_2)T_2^{N-1}f,\,
T_1^{p-1}T_2^{q-1}g
\right\rangle \\
&=\left\langle
(T_1^*T_1+T_2^*T_2-I)T_2^{N-1}f,\,
T_1^{p-1}T_2^{q-1}g
\right\rangle \\
&=0.
\end{aligned}
\]
\end{enumerate}
\textbf{3.} For $m=N$ and $n=1$, similarly, we obtain
$\langle T_1^NT_2f,\,
T_1^pT_2^qg\rangle =0$ for all $(p,q)\in {P}_{N+j}$, where $j\geq 2$.\\
\textbf{4.} Finally, let $m+n=N+1$ and $p+q\geq N+2$ with $m\geq 2$.
Then either $p-m>0$ or $q-n>0$. Assume that $p-m>0$.
Since $T_1$ is a $2$-isometry,
\[
(T_1^*)^mT_1^m
-
2(T_1^*)^{m-1}T_1^{m-1}
+
(T_1^*)^{m-2}T_1^{m-2}
=
0.
\]

Therefore, $\langle T_1^mT_2^nf,\,
T_1^pT_2^qg\rangle
=
\left\langle
(T_1^*)^mT_1^mT_2^nf,\,
T_1^{p-m}T_2^qg
\right\rangle
=0.$
The case $q-n>0$ can be treated similarly.
\item[\textbf{Case 2:}] Suppose $j=0$.
Since $m+n=p+q=N+1,$ either $m<p$ and $n>q$, or $m>p$ and $n<q$.
Without loss of generality, assume that $m<p$ and $n>q$.\\
\textbf{1.} Let $m=0,$  $n=N+1$. Then $\langle T_2^{N+1}f,\,
T_1^pT_2^qg\rangle
=
\langle T_1^*T_2^{N+1}f,\,
T_1^{p-1}T_2^qg\rangle
=0.$\\
\textbf{2.} Let $m=1$ and $n=N$.
\begin{enumerate}[label=\alph*)]
\item If $q=0$, then $\langle T_1T_2^Nf,\,
T_1^{N+1}g\rangle
=
\langle T_1T_2^{N-1}f,\,
T_2^*T_1^{N+1}g\rangle
=0.$
\item If $q\geq 1$, then using the toral $2$-isometry relation  \[T_2^{*q}T_1^*T_1T_2^q = T_1^*T_1+q\,(T_2^*T_2-I),\] (see \eqref{Power toral $2$-isometry condition}), we obtain
\[
\begin{aligned}
\hspace{1cm} \langle T_1T_2^Nf,\,
T_1^pT_2^qg\rangle
&=
\left\langle
T_2^{*q}T_1^*T_1T_2^q\,T_2^{N-q}f,\,
T_1^{p-1}g
\right\rangle \\
&=
\left\langle
\big(T_1^*T_1+q(T_2^*T_2-I)\big)
T_2^{N-q}f,\,
T_1^{p-1}g
\right\rangle \\
&=0.
\end{aligned}
\]
\end{enumerate}
\textbf{3.} For $m=N$ and $n=1$, similarly, we have $\langle T_1^NT_2f,\,
T_1^pT_2^qg\rangle
=0.$\\
\textbf{4.} Now let $m\geq 2$.
Using the $2$-isometry relation \[(T_1^*)^mT_1^m =
2(T_1^*)^{m-1}T_1^{m-1}-(T_1^*)^{m-2}T_1^{m-2},\] we obtain,
\begin{eqnarray*}
    \hspace{3.2cm} \langle T_1^mT_2^nf,\,
T_1^pT_2^qg\rangle
&=&
\left\langle
(T_1^*)^mT_1^mT_2^nf,\,
T_1^{p-m}T_2^qg
\right\rangle \\
&=&
\left\langle
\Big(
2(T_1^*)^{m-1}T_1^{m-1}
-
(T_1^*)^{m-2}T_1^{m-2}
\Big)
T_2^nf,\,
T_1^{p-m}T_2^qg
\right\rangle \\
&=&0.
\end{eqnarray*}
\textbf{5.} For $n\geq 2$, similarly, one obtains $\langle T_1^mT_2^nf,\,
T_1^pT_2^qg\rangle
=0.$
\end{itemize}
Hence, by induction, $T_1^mT_2^n(\ker \mathbf{T}^*) \perp T_1^pT_2^q(\ker \mathbf{T}^*)$ whenever $(m,n)\neq (p,q).$
 This completes the proof.
\end{proof}
Using Theorem\ref{Orthogonality}, we define $H_{11}
=
\bigoplus_{m,n\geq 0}
T_1^mT_2^n(\ker \mathbf{T}^*).$

In this subsection \ref{Subsection of operator-valued multishift} we will discuss about operator-valued multishift. 
\subsection{\textbf{Operator-valued multishift: }}\label{Subsection of operator-valued multishift}
To prove Theorem \ref{Main theorem}, we require the notion of operator-valued multishift, which we recall in this subsection. 
Let $\{ \mathcal{H}_\alpha : \alpha \in \mathbb{Z}_{+}^d \}$ be a multisequence of complex separable Hilbert spaces and let
$\mathcal{H}$ = $\bigoplus_{\alpha \in \mathbb{Z}_{+}^d} \mathcal{H}_\alpha$ be the orthogonal direct sum. Then $\mathcal{H}$ is a Hilbert space with respect to the inner product
\[
\langle x, y \rangle_{\mathcal{H}} := \sum_{\alpha \in \mathbb{Z}_{+}^d} \langle x_\alpha, y_\alpha \rangle_{\mathcal{H}_\alpha},
\]
for any $x = \oplus_{\alpha \in \mathbb{Z}_{+}^d} x_\alpha,\;
y = \oplus_{\alpha \in \mathbb{Z}_{+}^d} y_\alpha$ in $\mathcal{H}.$
If $\mathcal{H}_\alpha = \mathcal{H}$ for all $\alpha \in \mathbb{Z}_{+}^d$, then we denote $\mathcal{H} = \ell^2_{\mathcal{H}}(\mathbb{Z}_{+}^d)$. Let $\{ A^{(j)}_\alpha : \alpha \in \mathbb{Z}_{+}^d,\; j = 1, \dots, d \}$ be a multisequence of bounded linear operators $A^{(j)}_\alpha : \mathcal{H}_\alpha \to \mathcal{H}_{\alpha + \varepsilon_j}$, where $\varepsilon_j$ is the multi-index with $1$ at the $j$-th coordinate and $0$ elsewhere. An operator-valued multishift $\boldsymbol T = (T_1, \dots, T_d)$ on $\mathcal{H}$ with weights $\{A^{(j)}_\alpha\}$ is defined by
\[
D(T_j) := \Big\{ \oplus_{\alpha \in \mathbb{Z}_{+}^d} x_\alpha \in \mathcal{H} :
\sum_{\alpha \in \mathbb{Z}_{+}^d} \| A^{(j)}_\alpha x_\alpha \|^2 < \infty \Big\},
\]
and
\[
T_j \left( \oplus_{\alpha \in \mathbb{Z}_{+}^d} x_\alpha \right)
:= \oplus_{\alpha \in \mathbb{Z}_{+}^d} A^{(j)}_{\alpha - \varepsilon_j} \, x_{\alpha - \varepsilon_j},
\quad j = 1, \dots, d.
\]
If $\alpha_j = 0$ for some $\alpha \in \mathbb{Z}_{+}^d$, then we interpret $A^{(j)}_{\alpha - \varepsilon_j} = 0,\ 
x_{\alpha - \varepsilon_j} = 0, \
\mathcal{H}_{\alpha - \varepsilon_j} = \{0\}$. More generally, if any coordinate of $\alpha$ is negative, we set $\mathcal{H}_\alpha = \{0\}, \
A^{(j)}_\alpha = 0, \
x_\alpha = 0.$ 
Each $T_j$ is a densely defined linear operator on $\mathcal{H}$. Moreover, $T_j$ is bounded if and only if
\begin{eqnarray}\label{Uniformly Bounded Weights}
\sup_{\alpha \in \mathbb{Z}_{+}^d} \| A^{(j)}_\alpha \| < \infty.
\end{eqnarray}
Further, $T_i$ commutes with $T_j$ if and only if
\begin{eqnarray}\label{Commutativity-weights}
A^{(i)}_{\alpha + \varepsilon_j} A^{(j)}_\alpha
=
A^{(j)}_{\alpha + \varepsilon_i} A^{(i)}_\alpha
\quad \text{for all } \alpha \in \mathbb{Z}_{+}^d.
\end{eqnarray}
Equivalently, the following diagram commutes:
\[
\begin{array}{ccc}
\mathcal{H}_\alpha & \xrightarrow{A^{(j)}_\alpha} & \mathcal{H}_{\alpha + \varepsilon_j} \\
\downarrow{A^{(i)}_\alpha} &  & \downarrow{A^{(i)}_{\alpha + \varepsilon_j}} \\
\mathcal{H}_{\alpha + \varepsilon_i} & \xrightarrow{A^{(j)}_{\alpha + \varepsilon_i}} & \mathcal{H}_{\alpha + \varepsilon_i + \varepsilon_j}
\end{array}
\]
We say that $\boldsymbol T$ is a commuting operator-valued multishift if the weights satisfy \eqref{Uniformly Bounded Weights} and  \eqref{Commutativity-weights}. For more details of operator-valued multishifts, the reader is referred to  \cite{R.G} and \cite{R.S}.

Now we prove that a pair of left-inverse commuting analytic toral $2$-isometries satisfying the joint kernel condition is unitarily equivalent to a $2$-variable operator-valued weighted shift with invertible weights, whose initial two weights are positive.

\begin{theorem}\label{Model}
  If $\boldsymbol T:=(T_1,T_2)$ be left-inverse commuting analytic  toral $2$-isometry with joint kernel condition then $\boldsymbol T$ is unitary equivalent to a toral $2$-isometric operator-valued weighted shifts on $\ell^2_{\mathbb{Z}_{+}^2}(\ker \mathbf{T}^*)$ with invertible operator weights $\{W_{I}^{(j)}:j=1,2\}_{I\in \mathbb{Z}_{+}^2}$ such that the initial two weights $W_{0,0}^{(1)}$ and $W_{0,0}^{(2)}$ are positive. 
  \end{theorem}
\begin{proof}
For each $(m,n) \in \mathbb{Z}_+^2,$
let $\mathcal{M}_{m,n} := T_1^m T_2^n \big( \ker T_1^* \cap \ker T_2^* \big).$
Then, note that $\mathcal{M}_{m,n} \neq \{0\}$ for all $(m,n) \in \mathbb{Z}_+^2$ and $\mathcal M_{0,0}=\ker \boldsymbol{T^*}$.
For each $(m,n) \in \mathbb{Z}_+^2,$ define
\[
\Lambda_{m,n}^{(1)} := T_1\big|_{\mathcal{M}_{m,n}} 
: \mathcal{M}_{m,n} \to \mathcal{M}_{m+1,n},
\]
\[
\Lambda_{m,n}^{(2)} := T_2\big|_{\mathcal{M}_{m,n}} 
: \mathcal{M}_{m,n} \to \mathcal{M}_{m,n+1}.
\]
The maps $\Lambda_{m,n}^{(1)}$ and $\Lambda_{m,n}^{(2)}$ are linear homeomorphisms. 
Hence, for every $(m,n) \in \mathbb{Z}_+^2$, the Hilbert spaces 
$\mathcal{M}_{m,n}$ and $\mathcal{M}_{0,0}$ are unitarily equivalent 
(see \cite[Problem 56]{HA}). For each $(m,n) \in \mathbb{Z}_+^2$, let $V_{m,n} : \mathcal{M}_{m,n} \to \mathcal{M}_{0,0}$ be a unitary isomorphism. By Theorem~\ref{Orthogonality},  we have
$H_{11} = \bigoplus_{m,n \geq 0} \mathcal{M}_{m,n}$. Define a unitary operator $V : H_{11} \longrightarrow \ell^2_{\mathbb{Z}_{+}^2}(\mathcal{M}_{0,0})$ by the rule
\[
V\!\left( \oplus_{m,n \geq 0} h_{m,n} \right)
=
\oplus_{m,n \geq 0} V_{m,n} h_{m,n}.
\]
Let $A_1 := T_1|_{H_{11}}$ and $A_2 := T_2|_{H_{11}}$. Then $(A_1, A_2)$ is a toral $2$-isometry. Moreover,
\[
V A_1 = S_1 V 
\quad \text{and} \quad 
V A_2 = S_2 V,
\]
where $(S_1, S_2) \in B\big(\ell^2_{\mathbb{Z}_{+}^2}(\ker \mathbf{T^*})\big)$
is an operator-valued weighted shift with weights
\[
W_{m,n}^{(1)} = V_{m+1,n}\,\Lambda_{m,n}^{(1)}\,V_{m,n}^{-1}, 
\quad
W_{m,n}^{(2)} = V_{m,n+1}\,\Lambda_{m,n}^{(2)}\,V_{m,n}^{-1}, \quad (m,n) \in \mathbb{Z}_+^2.
\]
In particular, define
\[
V_{1,0} : T_1(\ker \mathbf{T}^*) \to \ker \mathbf{T}^*, 
\quad V_{1,0}(T_1 f) = (T_1^* T_1)^{1/2} f,
\]
\[
V_{0,1} : T_2(\ker \mathbf{T}^*) \to \ker \mathbf{T}^*, 
\quad V_{0,1}(T_2 f) = (T_2^* T_2)^{1/2} f,
\]
for $f \in \ker \mathbf{T}^*$. Then both $V_{1,0}$ and $V_{0,1}$ are unitary operators. Consequently, the initial weights of the shift are $W_{0,0}^{(1)} = |T_1|$ and $W_{0,0}^{(2)} = |T_2|$ on $\ker \mathbf{T}^*.$ This completes the proof.
\end{proof}
\begin{definition}\label{Moment Product}
Let $(T_1,T_2)$ be a commuting operator-valued weighted shift with operator weights 
$\{W_{I}^{(j)}: j=1,2\}_{I\in \mathbb{Z}_{+}^2}$. 
For $m,n\in \mathbb Z_+$, define
\begin{eqnarray}\label{P-mn}
    \mathcal{P}_{m,n}
=
W_{m-1,n}^{(1)}
W_{m-2,n}^{(1)}
\cdots
W_{0,n}^{(1)}
W_{0,n-1}^{(2)}
\cdots
W_{0,0}^{(2)},
\end{eqnarray}
with the convention that $\mathcal{P}_{0,0}=I$.
\end{definition}
\begin{definition}
The class of all operator-valued weighted shifts 
$(T_1,T_2)$ on $\ell_{\mathbb{Z}_{+}^2}^{2}(\mathcal H)$ 
with invertible operator weights is denoted by 
$\ell_{IW}^{2}(\mathcal H)$.
Further, the subclass consisting of those shifts for which 
$\mathcal{P}_{m,n}$ is positive for every $(m,n)\in \mathbb{Z}_{+}^{2}$ 
is denoted by $\ell_{IWP}^{2}(\mathcal H)$.
\end{definition}
Before proving  Theorem \ref{Main theorem}, we record the following structural proposition for operator-valued multishifts with invertible weights. This result allows us to reduce the study of such multishifts to a canonical form on the class $\ell^{2}_{IWP}(\mathcal{H})$, which will be used repeatedly in the subsequent sections.
\begin{proposition}
 Suppose $(T_1,T_2)$ be a tuple of commuting operator-valued multishift with invertible operator weights $\{W_{I}^{(j)}:j=1,2\}_{I\in \mathbb{Z}_{+}^2}$ such that $W_{0,0}^{(1)}$ and $W_{0,0}^{(2)}$ are positive. Then $(T_1,T_2)$ is unitarily equivalent to commuting operator-valued weighted shifts $(\tilde{T_1},\tilde{T_2)}$
 in the class of $\ell^{2}_{IWP}(\mathcal{H})$ with weights $\{\tilde{W}_{I}^{(j)}:j=1,2\}_{I\in \mathbb{Z}_{+}^2}$
  where as $\tilde W_{0,0}^{(1)}=W_{0,0}^{(1)}$ and $\tilde W_{0,0}^{(2)}=W_{0,0}^{(2)}$.
\end{proposition}
\begin{proof}
 For each $(m,n)\in\mathbb Z_+^2,$ let $\mathcal{P}_{m,n}=U_{m,n}|\mathcal{P}_{m,n}|$ be a polar decomposition of $\mathcal{P}_{m,n}$. 
 Define a unitary $U:\ell_{\mathbb{Z}_{+}^2}^2(\mathcal{H}) \to  \ell_{\mathbb{Z}_{+}^2}^2(\mathcal{H})$ as 
 $$U(e_{m,n} \otimes h)=e_{m,n}\otimes U_{m,n}h \quad h\in \mathcal H, m,n\in\mathbb Z_+.$$
 Define $\tilde{T_1}=U^*T_1U$ and $\tilde{T_2}=U^*T_2U$. Then, note that $h\in \mathcal H,$ and $m,n\in\mathbb Z_+,$
 \begin{align*}
\tilde{T_1}(e_{m,n} \otimes h)
&= U^* T_1 U (e_{m,n} \otimes h) \\
&= U^* T_1 (e_{m,n} \otimes U_{m,n} h) \\
&= U^* (e_{m+1,n} \otimes W_{m,n}^{(1)} U_{m,n} h) \\
&= e_{m+1,n} \otimes U_{m+1,n}^* W_{m,n}^{(1)} U_{m,n} h
\end{align*}
and \begin{align*}
\tilde{T_2}(e_{m,n} \otimes h)
&= U^* T_2 U (e_{m,n} \otimes h) \\
&= U^* T_2 (e_{m,n} \otimes U_{m,n} h) \\
&= U^* (e_{m,n+1} \otimes W_{m,n}^{(2)} U_{m,n} h) \\
&= e_{m,n+1} \otimes U_{m,n+1}^* W_{m,n}^{(2)} U_{m,n} h.
\end{align*}
Therefore $\tilde{W}_{m,n}^{(1)}=U_{m+1,n}^* W_{m,n}^{(1)} U_{m,n}$ and 
 $\tilde{W}_{m,n}^{(2)}=U_{m,n+1}^* W_{m,n}^{(2)} U_{m,n}.$
 With simple calculation, it can be shown that $\mathcal{P}_{m,n}^{'}=|\mathcal{P}_{m,n}|$ $\geq 0.$
\end{proof}
Here we derive a formula for the weights of toral $2$-isometric weighted shifts belonging to the class $\ell^2_{IWP}(\mathcal{H})$.
\begin{theorem}\label{Representation of weights}
  Let $(T_1,T_2)$ be a toral $2$-isometric weighted shifts with invertible operator weights $\{W_{I}^{(j)}:j=1,2\}_{I\in \mathbb{Z}_{+}^2}$ and with in the class of $\ell^2_{IWP}(\mathcal{H})$. Then, for each $(m,n)\in\mathbb Z_+^2,$
  \begin{eqnarray}\label{Formula for horizontal weights}
      W_{m,n}^{(1)} &=&
\sqrt{(m+1)A
+ nB + I}
\,
\left(
\sqrt{mA
+ nB+ I}
\right)^{-1}
  \end{eqnarray}
and 
\begin{eqnarray}\label{Formula for vertical weights}
    W_{m,n}^{(2)} =
\sqrt{(n+1)B
+ mA + I}
\,
\left(
\sqrt{mA
+ nB + I}
\right)^{-1},
\end{eqnarray}
where $ A=\big((W_{0,0}^{(1)})^2 - I\big)$ and $B=\big((W_{0,0}^{(2)})^2 - I\big).$
\end{theorem}
\begin{proof}
Let $(m,n)\in\mathbb Z_+^2$ and  $h\in \mathcal{H}$ with $\|h\|=1.$ 
Define $\gamma_{m,n}^h:=\|T_1^mT_2^n(e_{0,0}\otimes h)\|^2.$ 
Note that  $\gamma_{m,n}^h=\|\mathcal{P}_{m,n}h\|^2$. Let $f=T_1^{m}T_2^n(e_{0,0}\otimes h)$. Then by the definition of toral $2$-isometry $\|T_i^2f\|^2-2\|T_if\|^2+\|f\|^2=0,$ for $i=1,2$ (see Definition \ref{Toral $2$ isometry}), we have  
\begin{eqnarray}\label{For $T_1$}
\gamma_{m+2,n}^h - 2\gamma_{m+1,n}^h + \gamma_{m,n}^h = 0.
\end{eqnarray}
\begin{eqnarray}\label{For $T_2$}
\gamma_{m,n+2}^h - 2\gamma_{m,n+1}^h + \gamma_{m,n}^h = 0.
\end{eqnarray}
 Again by the mixed condition for toral $2$ isometry $\|f\|^2-\|T_1f\|^2-\|T_2f\|^2+\|T_1T_2f\|^2=0,$ we have 
\begin{eqnarray}\label{For $T_1,T_2$}
\gamma_{m+1,n+1}^h - \gamma_{m,n+1}^h-\gamma_{m+1,n}^h + \gamma_{m,n}^h = 0.
\end{eqnarray}
From \eqref{For $T_1$}, $\gamma_{m,n}^h$ is an arithmetic progression in $m$, hence
\[
\gamma_{m,n}^h = a_n + m b_n
\]
where, 
$a_n=\gamma_{0,n}^h$ and $b_n=\gamma_{1,n}^h-\gamma_{0,n}^h.$
Plugging the value of $\gamma_{m,n}^h$ in \eqref{For $T_2$}, we have 
\[
(a_{n+2} + m b_{n+2}) - 2(a_{n+1} + m b_{n+1}) + (a_{n} + m b_{n}) = 0.
\]
Rewriting the above equation, we get 
\[
(a_{n+2} - 2a_{n+1} + a_{n}) + m(b_{n+2} - 2b_{n+1} + b_{n}) = 0.
\]
Thus, $a_{n+2} - 2a_{n+1} + a_{n} = 0,\quad
b_{n+2} - 2b_{n+1} + b_{n} = 0.$
Hence, we get 
$a_n = \alpha + \delta n,\ 
b_n = \beta + \epsilon n,$
where $\alpha=a_0$, $\delta=a_1-a_0$ and $\beta=b_0$, $\epsilon=b_1-b_0.$
Hence,  $$\gamma_{m,n}^h = \alpha + \beta m + \delta n + \epsilon mn.$$
Substitute these values $\gamma_{m+1,n+1}^h$, $\gamma_{m,n+1}^h$, $\gamma_{m+1,n}^h $ and $\gamma_{m,n}^h$ in \eqref{For $T_1,T_2$}, we have $\epsilon=0$ and therefore $\gamma_{m,n}^h = \alpha + \beta m + \delta n$.
Since $\gamma_{0,0}^h=1,$ it follows that  $\gamma_{m,n}^h = 1 + \beta m + \delta n$,
for all $(m,n)\in \mathbb{Z}_+^2.$ 
After substituting $m=1,n=0$ and $m=0,n=1$  in the above equation and for all $h\in \mathcal{H}$ we have that 
\begin{eqnarray}\label{Formula for P-mn}
  \mathcal{P}_{m,n}^2={m\big((W_{0,0}^{(1)})^2 - I\big)
+ n\big((W_{0,0}^{(2)})^2 - I\big) + I}, \quad m,n \in \mathbb{Z}_{+}. 
\end{eqnarray}
 Since $W_{m,n}^{(1)} = \mathcal{P}_{m+1,n}\,(\mathcal{P}_{m,n})^{-1}$ and $W_{m,n}^{(2)} = \mathcal{P}_{m,n+1}\,(\mathcal{P}_{m,n})^{-1},$ 
substituting the value of $\mathcal{P}_{m,n}$ from \eqref{Formula for P-mn}, we have the representations \eqref{Formula for horizontal weights} and \eqref{Formula for vertical weights}. 
\end{proof}
\begin{remark}
For a commuting operator-valued $2$-variable weighted shift $\mathbf{T}=(T_1,T_2)$ with weights $\{W_{I}^{(j)}:j=1,2\}_{I\in\mathbb Z_+^2}$ in $\mathcal{H},$  we set 
\begin{eqnarray}\label{Moment weights product}
\mathcal{T}_{[{(m,n),(i,j)}]} {=}
\begin{cases}
W^{(1)}_{m-1,n} \cdots W^{(1)}_{i,n} \, W^{(2)}_{i,n-1} \cdots W^{(2)}_{i,j} 
& \text{if } m \ge i,\; n \ge j. \\
I 
& \text{if } m = i,\; n = j.
\end{cases}
\end{eqnarray}
For $\beta=(\beta_1,\beta_2)\in \mathbb Z_+^2$ with $|\beta|=\beta_1+\beta_2=k$, define
\begin{eqnarray}\label{Toral $k$-isometry}
\Delta_{k,s,t}^{\mathrm{toral}}(f)
=
\sum_{\alpha \leq \beta}
(-1)^{|\alpha|}
\binom{\beta}{\alpha}
\left\|
\mathcal{T}_{[{(s+\alpha_1,\; t+\alpha_2),(s,t)}]}f
\right\|^2,
\end{eqnarray}
for all $s,t\geq 0$, $f\in \mathcal H$, and $\binom{\beta}{\alpha}
=
\prod_{i=1}^{2}
\binom{\beta_i}{\alpha_i},$
where $\alpha=(\alpha_1,\alpha_2),$ and $\beta=(\beta_1,\beta_2).$
The notation used above is adopted from \cite[Remark 2.4]{J.Z}.
\end{remark}
The following proposition provides an equivalent characterization of toral $k$-isometries for $2$-variable operator-valued weighted shift.
\begin{proposition}\label{Weighted toral 2-isometry}
Let $\mathbf{T}=(T_1,T_2)$ be a commuting operator-valued multishift. Then $\mathbf{T}$ is a toral $k$-isometry if and only if $\Delta_{k,s,t}^{\mathrm{toral}}(f)=0$
for all $s,t\geq 0$ and for all $f\in \mathcal H$.
\end{proposition}
\begin{proof}
 The proof is very much similar to \cite[Proposition 2.5 (i)]{J.Z} and \cite[Proposition 1.3]{Z.J}.   
\end{proof}
We conclude this section by describing cyclic analytic toral $2$-isometry satisfying the joint kernel condition. Namely by \cite[Theorem 2.4]{S.B}, a cyclic analytic  toral $2$-isometry $\mathbf{T}=(T_1,T_2)$ is unitarily equivalent to the tuple of operators $(M_{z_1},M_{z_2})$ on a Dirichlet-type space $\mathcal{D}(\mu_1,\mu_2)$  if and only if $\ker\mathbf{T}^*$ has wandering subspace property. The Hilbert space $\mathcal{D}(\mu_1,\mu_2)$, as introduced in \cite{S.B}, consists of all analytic functions $f$ on the bidisc $\mathbb{D}^2$ such that $\mathcal{D}_{\mu_1,\mu_2}(f) <\infty,$  where 
\[
\begin{aligned}
\mathcal{D}_{\mu_1,\mu_2}(f)
&= \sup_{0<r<1} \int_{\mathbb{T}} \int_{\mathbb{D}} 
\left|\partial_1 f(z_1, re^{i\theta})\right|^2 
\, P_{\mu_1}(z_1)\, dA(z_1)\, d\theta \\
&\quad + \sup_{0<r<1} \int_{\mathbb{T}} \int_{\mathbb{D}} 
\left|\partial_2 f(re^{i\theta}, z_2)\right|^2 
\, P_{\mu_2}(z_2)\, dA(z_2)\, d\theta,
\end{aligned}
\]  
the symbols $\partial_1 f$ and $\partial_2f$ stand for the partial derivative of $f$ with respect to $z_1$ and $z_2$ respectively, $dA$ denotes the normalized Lebesgue area measure on $\mathbb{D}$ and for $i=1,2$ the function $P_{\mu_{i}}$ is the positive harmonic function on $\mathbb{D}$ defined by
\[P_{\mu_{i}}(z_{i})=\frac{1}{2\pi}\int_{[0,2\pi)} \frac{1-|z_{i}|^2}{|e^{i\theta}-z_{i}|^2}d\mu_{i}(\theta). \]  
The inner product $\langle \cdot, \cdot\rangle_{\mu_1,\mu_2}$ of $\mathcal{D}(\mu_1,\mu_2)$ is given by 
\[
\begin{aligned}
\langle f,g\rangle_{\mu_1,\mu_2}
&= \langle f,g\rangle_{H^2(\mathbb{D}^2)} \\
&\quad + \lim_{r\to 1^{-}} \int_{\mathbb{T}} \int_{\mathbb{D}} 
\partial_1 f(z_1, re^{i\theta}) \overline{\partial_1 g(z_1, re^{i\theta})} 
\, P_{\mu_1}(z_1)\, dA(z_1)\, d\theta \\
&\quad + \lim_{r\to 1^{-}} \int_{\mathbb{T}} \int_{\mathbb{D}} 
\partial_2 f(re^{i\theta}, z_2) \overline{\partial_2 g(re^{i\theta}, z_2)} 
\, P_{\mu_2}(z_2)\, dA(z_2)\, d\theta.
\end{aligned}
\]
\begin{proposition}\label{Example}
 Under the above assumption $\mathbf{M}_{z}=(M_{z_1},M_{z_2})$ on $\mathcal{D}(\mu_1,\mu_2)$ satisfies the joint kernel condition if and only if $\mu_1=\alpha_1m$  and $\mu_2=\alpha_2m$  for some $\alpha_1,\alpha_2 \in \mathbb{R}_{+},$ where $m$ is the Lebesgue measure on $[0,2\pi).$
\end{proposition}
\begin{proof}
 Suppose  $\mu_1=\alpha_1m$  and $\mu_2=\alpha_2m.$ Using the Poisson integral formula, we have $P_{\mu_1}(re^{i\phi})=\alpha_1$ and $P_{\mu_2}(re^{i\psi})=\alpha_2,$ 
 where $\phi, \psi \in \mathbb{R}$, $r\in [0,1)$. Using the inner product on $\mathcal{D}(\mu_1,\mu_2),$ the bisequence $\{e_{m,n}\}_{m,n=0}^{\infty}$ define by $e_{m,n}(z_1,z_2)=\frac{1}{\sqrt{1+m\alpha_1+n\alpha_2}}z_1^mz_2^n$ is an orthonormal basis of $\mathcal{D}(\mu_1,\mu_2)$. Since \[M_{z_1}e_{m,n}=\sqrt{\frac{1+(m+1)\alpha_1+n\alpha_2}{1+m\alpha_1+n\alpha_2}}e_{m+1,n},\]
 \[M_{z_2}e_{m,n}=\sqrt{\frac{1+(n+1)\alpha_2+m\alpha_1}{1+m\alpha_1+n\alpha_2}}e_{m,n+1}\]
for all $m,n \in \mathbb{Z}_{+}$, we deduce that $\mathbf{M}_{z}=(M_{z_1},M_{z_2})$
is unitarily equivalent to the two variable weighted shifts with weighted sequences $\{\zeta_{m,n}^{(j)}(\lambda_1,\lambda_2):j=1,2\}_{m,n=0}^{\infty}$,  where $\lambda_1=\sqrt{1+\alpha_1}$, $\lambda_2=\sqrt{1+\alpha_2}$ and 
\[\zeta_{m,n}^{(1)}(\lambda_1,\lambda_2)=\sqrt{\frac{1+(m+1)\alpha_1+n\alpha_2}{1+m\alpha_1+n\alpha_2}}\]
\[\zeta_{m,n}^{(2)}(\lambda_1,\lambda_2)=\sqrt{\frac{1+(n+1)\alpha_2+m\alpha_1}{1+m\alpha_1+n\alpha_2}}.\]
As a consequence, $(M_{z_1},M_{z_2})$ satisfies the joint kernel condition.
Conversely, Suppose that $(M_{z_1},M_{z_2})$ on $\mathcal{D}(\mu_1,\mu_2)$ satisfies the joint kernel conditions, where $\mu_1,\mu_2$ are finite positive Borel measures on $[0,2\pi)$.
Since $(M_{z_1},M_{z_2})$ is a pair of cyclic analytic toral $2$-isometries such that $\dim(\ker \mathbf{M}_z^*)=1$ (see \cite[Corollary 3.9]{S.B}), it follows by Theorem $\ref{Representation of weights}$ that  $(M_{z_1},M_{z_2})$ is unitarily equivalent to a toral $2$-isometric $2$-variable weighted shift say $(S_1,S_2)$ with weights $\{\zeta_{m,n}^{(1)}(\lambda_1,\lambda_2)\}_{m,n=0}^{\infty}$ and $\{\zeta_{m,n}^{(2)}(\lambda_1,\lambda_2)\}_{m,n=0}^{\infty}$ . In view of the previous paragraph, $(S_1,S_2)$ is unitarily equivalent to the operator tuple $(M_{z_1,\alpha_1m},M_{z_2,\alpha_2m})$ of multiplication tuple  by the coordinate function $z_1,z_2$ on $\mathcal{D}(\alpha_1m,\alpha_2m)$ where $\alpha_1=\lambda_1^2 -1$, $\alpha_2=\lambda_2^2 -1$  $\in \mathbb{R}_{+}.$ Applying \cite[Proposition 6.4]{S.B}, we conclude that $\mu_1=\alpha_1m$ and $\mu_2=\alpha_2m$.
\end{proof}
Define $\Pi_{0,0} : \mathcal{H} \to \ell^2(\mathbb{Z}_+^2)\otimes \mathcal{H}$ by
$\Pi_{0,0}(h) = e_{0,0} \otimes h$ for any $h \in \mathcal{H}.$
This map places the vector $h$ at the coordinate $(0,0)$ and zero elsewhere, that is,
\[
\Pi_{0,0}(h)(m,n) =
\begin{cases}
h & \text{if } (m,n) = (0,0), \\
0 & \text{otherwise}.
\end{cases}
\]
Note that $\Pi_{0,0}$ is an isometry. 
More generally, for $(m,n) \in \mathbb{Z}_+^2$ and $h\in\mathcal H$, define $\Pi_{m,n}(h) = e_{m,n} \otimes h.$
Note that
\[
\langle \Pi_{m,n}h,\ \Pi_{p,q}k \rangle =
\begin{cases}
\langle h,k \rangle & \text{if } (m,n) = (p,q), \\
0 & \text{otherwise}.
\end{cases}
\]
We now discuss the connection between the subnormality of operator-valued multishifts and Stieltjes moment sequences.
\begin{theorem}\label{Subnormality criteria}
  Let $\mathbf{T}=(T_1,T_2)$ be $2$-variable commuting operator-valued weighted shifts with weights $\{W_{I}^{(j)}:j=1,2\}_{I\in \mathbb{Z}_{+}^2}$ in $\mathcal{H}$ such that $\ker({W_I^{(j)}}^*)=\{0\}$ for $j=1,2$ and for all $I\in \mathbb{Z}_{+}^2.$ Then the tuple $\mathbf{T}=(T_1,T_2)$ is subnormal if and only if $\{\|\mathcal{T}_{[(m,n),(0,0)]}f\|^2\}$ is a Stieltjes moment sequence for all $f\in \mathcal{H}$.
\end{theorem}
\begin{proof}
Let  $\mathbf{T} = (T_1, T_2)$  be subnormal. Then there exist a Hilbert space  
$K \supseteq \ell^2(\mathbb{Z}_{+}^2)\otimes \mathcal{H}$  and commuting normal operators  $N_1, N_2$ on $K$ with $\ell^2(\mathbb{Z}_{+}^2)\otimes \mathcal{H}$ is invariant under each of $N_j$ such that 
\[
T_i = N_i\big|_{\ell^2(\mathbb{Z}_{+}^2)\otimes \mathcal{H}}  \quad i = 1,2.
\]
Since $N_1$ and $N_2$ are commuting normal operators, by the joint spectral theorem there exists 
a spectral measure $E(\cdot)$ on $\mathbb{C}^2$ such that for each $m,n\in\mathbb Z_+,$
\[
N_1^m N_2^n = \int s^m t^n \, dE(s,t).
\]
Let $f \in \mathcal{H}$. Then $T_1^m T_2^n (e_{(0,0)} \otimes f) = e_{(m,n)} \otimes \mathcal{T}_{[(m,n),(0,0)]} f,$
and $N_1^m N_2^n (e_{(0,0)} \otimes f)
= T_1^m T_2^n (e_{(0,0)} \otimes f).$
Hence,
\[
\| N_1^m N_2^n (e_{(0,0)} \otimes f) \|^2
= \| T_1^m T_2^n (e_{(0,0)} \otimes f) \|^2
= \| \mathcal{T}_{[(m,n),(0,0)]} f \|^2.
\]
Also,
\[
\| N_1^m N_2^n (e_{(0,0)} \otimes f) \|^2
= \int |s|^{2m} |t|^{2n} \, d\langle E(s,t)(e_{(0,0)} \otimes f), (e_{(0,0)} \otimes f) \rangle.
\]
Define $\mu_f(\cdot) := \langle E(\cdot)(e_{(0,0)} \otimes f), (e_{(0,0)} \otimes f) \rangle.$
Therefore,
\[
\| \mathcal{T}_{[(m,n),(0,0)]} f \|^2
= \int |s|^{2m} |t|^{2n} \, d\mu_f(s,t).
\]
Define a map: $\Phi :\mathbb{C}^2 \to \mathbb{R}_+^2$ as $\Phi(s,t)=(|s|^2 ,|t|^2)$. Let $u=|s|^2$, $v=|t|^2$ and define a new measure $\nu(B):=\mu_f(\Phi^{-1}(B))$, for all Borel subset $B$ of  $\mathbb{R}_+^2$

\[
\| \mathcal{T}_{[(m,n),(0,0)]} f \|^2
= \int u^mv^n \, d\nu(u,v).
\]
This imply,  $\left\{ \| \mathcal{T}_{[(m,n),(0,0)]} f \|^2 \right\}$ is a $2$-variable Stieltjes moment sequence.

Conversely, assume that for every $f\in \mathcal{H}$,  $\left\{ \| \mathcal{T}_{[(m,n),(0,0)]} f \|^2 \right\}$ is a Stieltjes moment sequence. Let $f\in \mathcal{H}$.
There exists a positive Borel measure $\mu_f$ on $\mathbb{R}_{+}^2$ such that 
\[
\| \mathcal{T}_{[(m,n),(0,0)]} f \|^2
= \int s^{m} t^{n} \, d\mu_f(s,t).
\]
Choose $X=\Pi_{0,0}\mathcal{H}$  and let $f\in \mathcal{H}$  then 
\[
\begin{aligned}
\langle B_{\mathbf{T}}((m,n),(p,q))& \Pi_{0,0} f,\; \Pi_{0,0} f \rangle
= \langle T_1^m T_2^n \Pi_{0,0} f,  \  T_1^p T_2^q \Pi_{0,0} f \rangle = \delta_{(m,n),(p,q)}\| \mathcal{T}_{[(m,n),(0,0)]} f \|^2\\
&=\frac{1}{(2\pi)^2}\int_{0}^{2\pi}\int_{0}^{2\pi}\int_{\mathbb{R}_{+}^2}(s^{1/2}e^{i\theta})^m \overline{(s^{1/2}e^{i\theta})^p}(t^{1/2}e^{i\phi})^n \overline{(t^{1/2}e^{i\phi})^q}d\mu_f(s,t)d\theta d\phi.
\end{aligned}
\]
Therefore using this we can prove that  $\langle B_{\mathbf{T}}(\cdot)h,h\rangle$ is positive definite over $\mathfrak{R}_2$ for all $h\in \Pi_{0,0}\mathcal{H}$, see \cite[ Introduction]{St}.
Since $\ker({W_{m,n}^{(1)}}^*)=\ker({W_{m,n}^{(2)}}^*)=\{0\}$
for any $(m,n)\in \mathbb{Z}_{+}^2,$ therefore we get $\ker(\mathcal{T}_{[(m,n),(0,0)]})^*=\{0\}.$ Hence $\overline{ran(\mathcal{T}_{[(m,n),(0,0)]})}$ =$\mathcal{H}$ 
and $$\bigvee_{m,n\geq 0} T_1^mT_2^n \Pi_{0,0}\mathcal{H} =\big\{(x_{m,n})_{m,n\in \mathbb{Z}_{+}} : x_{m,n}\in \overline{Ran (\mathcal{T}_{[(m,n),(0,0)]})}\big\}.$$ Thus it follows that $\ell^2(\mathbb {Z}_{+}^2)\otimes\mathcal{H}=$  $\bigvee_{m,n\geq 0}\{T_1^mT_2^n\Pi_{0,0}h:h\in \mathcal{H}\}$. The subnormality of $\mathbf{T}$ follows now from \cite[Theorem 3.1]{St} applied to $X=\Pi_{0,0}\mathcal{H}$.
\end{proof}
\section{\textbf{The Cauchy dual subnormality problem via joint kernel condition}}\label{Section on Cauchy Dual Subnormality}
In this section, we answer the Cauchy dual subnormality problem  in the affirmative for toral $2$-isometries that satisfy the joint kernel condition with commutativity of initial two weights. A bisequence $\gamma = \{\gamma_{m,n}\}_{n=0}^\infty \subseteq \mathbb{R}^2$
is said to be a {\it Hamburger (resp., Stieltjes, Hausdorff) moment bisequence} if there exists
a positive Borel measure $\mu$ on $\mathbb{R}^2$ (resp., $\mathbb{R}^2_+,$ $[0,1]\times [0,1]$) such that
\[
\gamma_{m,n} = \int t_1^mt_2^n \, d\mu(t_1,t_2) \quad \text{for every } m,n \in \mathbb{Z}_+.
\]
Such a $\mu$ is called a representing measure of the bisequence $\gamma$.  We refer the reader to \cite{C.B},\cite{B.S} for more information on moment problems. 



Before proving Theorem \ref{Main theorem}, we establish an integral representation for toral $2$-isometric operator-valued weighted shifts with in the class of $\ell^2_{IWP}(\mathcal{H})$. The following proposition shows that when the initial weights commute, every weight in the shift admits an explicit representation in terms of the joint spectral measure of the initial weights. This result will play a crucial role in the proof of the main theorem.
\begin{proposition}\label{Integral representation}
  Let $(T_1,T_2)$ be a toral $2$-isometric weighted shifts in the class of $\ell^2_{IWP}(\mathcal{H})$ with weights $\{W_{I}^{(j)}:j=1,2\}_{I\in\mathbb Z_+^2}$ such that  the initial weights $W_{0,0}^{(1)}$  and $W_{0,0}^{(2)}$ commute then
\begin{eqnarray}\label{Integral formula for horizontal weights}
    W_{m,n}^{(1)}=\int_{[1,a_1]\times [1,a_2]} \sqrt{\frac{1+(m+1)(x^2-1)+n(y^2-1)}{1+m(x^2-1)+n(y^2-1)}}dE(x,y)
\end{eqnarray}
and 
\begin{eqnarray}\label{Integral formula for vertical weights}
    W_{m,n}^{(2)}=\int_{[1,a_1]\times [1,a_2]} \sqrt{\frac{1+m(x^2-1)+(n+1)(y^2-1)}{1+m(x^2-1)+n(y^2-1)}}dE(x,y),
\end{eqnarray}
where $E$ is a joint spectral measure on $[1,a_1]\times [1,a_2]$ with $a_1=\|W_{0,0}^{(1)}\|$ and $a_2=\|W_{0,0}^{(2)}\|.$
\end{proposition}
\begin{proof}
 Since $W_{0,0}^{(1)}\geq I,$ $W_{0,0}^{(2)}\geq I,$ and $W_{0,0}^{(1)}$ and $W_{0,0}^{(2)}$ commute, by the joint spectral theorem for commuting bounded self-adjoint operators, there exists a joint spectral measure $E$ on $[1,a_1]\times [1,a_2]$ such that
 \[
 W_{0,0}^{(1)}=\int_{[1,a_1]\times [1,a_2]} x\, dE(x,y)
 \]
 and
 \[
 W_{0,0}^{(2)}=\int_{[1,a_1]\times [1,a_2]} y\, dE(x,y),
 \]
 Since $W_{0,0}^{(1)}$ and $W_{0,0}^{(2)}$ commute, the operators
 $(W_{0,0}^{(1)})^2$ and $(W_{0,0}^{(2)})^2$ also commute. Hence, by the functional calculus for commuting self-adjoint operators,
 \[
 (W_{0,0}^{(1)})^2
 =
 \int_{[1,a_1]\times [1,a_2]} x^2\, dE(x,y)
 \]
 and
 \[
 (W_{0,0}^{(2)})^2
 =
 \int_{[1,a_1]\times [1,a_2]} y^2\, dE(x,y).
 \]
 Therefore,\[
 I+(m+1)((W_{0,0}^{(1)})^2-I)+n((W_{0,0}^{(2)})^2-I)
 =
 \int_{[1,a_1]\times [1,a_2]}
 \Big(1+(m+1)(x^2-1)+n(y^2-1)\Big)\, dE(x,y),
 \]
 and 
 \[
 I+m((W_{0,0}^{(1)})^2-I)+n((W_{0,0}^{(2)})^2-I)
 =
 \int_{[1,a_1]\times [1,a_2]}
 \Big(1+m(x^2-1)+n(y^2-1)\Big)\, dE(x,y).
 \]
 Similarly we obtain that,\[
 I+m((W_{0,0}^{(1)})^2-I)+(n+1)((W_{0,0}^{(2)})^2-I)
 =
 \int_{[1,a_1]\times [1,a_2]}
 \Big(1+m(x^2-1)+(n+1)(y^2-1)\Big)\, dE(x,y).
 \]
 Using the multiplicativity property of the spectral integral, commutativity of spectral projection (cf.\cite[Preliminaries]{B.W}) and the inversion formula for bounded Borel functions and Theorem $\ref{Representation of weights}$, we have the representations stated in \eqref{Integral formula for horizontal weights} and \eqref{Integral formula for vertical weights}.
 Hence the proof is complete.
\end{proof}
\begin{remark}\label{Cauchy dual weights}
Let $\mathbf{T} = (T_1, T_2)$ be a $2$-variable commuting operator-valued weighted shift with invertible weights 
$\{W_{I}^{(j)}:j=1,2\}_{I\in \mathbb{Z}_{+}^2}$. For all $(m,n)\in \mathbb{Z}_{+}^2$ and $h\in \mathcal{H}$
the operators are defined by
\begin{align*}
T_1 (e_{m,n} \otimes h) &= e_{m+1,n} \otimes W^{(1)}_{m,n} h, \\ 
T_2 (e_{m,n} \otimes h) &= e_{m,n+1} \otimes W^{(2)}_{m,n} h.
\end{align*}
The adjoint satisfies
\begin{align*}
T_1^* (e_{m,n} \otimes h)
=
\begin{cases}
e_{m-1,n} \otimes W^{(1)*}_{m-1,n} h, & m \ge 1, \\
0, & m = 0,
\end{cases}
\end{align*}
and hence $T_1^* T_1 (e_{m,n} \otimes h)
= T_1^* \big( e_{m+1,n} \otimes W^{(1)}_{m,n} h \big) = e_{m,n} \otimes W^{(1)*}_{m,n} W^{(1)}_{m,n} h.$
A simple computation shows that $T_1' (e_{m,n} \otimes h)= e_{m+1,n} \otimes (W^{(1)*}_{m,n})^{-1} h$ and $T_2' (e_{m,n} \otimes h)=e_{m,n+1} \otimes (W^{(2)*}_{m,n})^{-1} h$ for all $(m,n) \in \mathbb{Z}_+^2$ and $h\in\mathcal H.$ As $(T_1,T_2)$ is commuting, therefore the  Cauchy dual of $\mathbf{T}$ is also commuting.
\end{remark}
Using \eqref{Moment weights product}, \eqref{Integral formula for horizontal weights} and \eqref{Integral formula for vertical weights}, we have
\begin{eqnarray}\label{Integral representation of moments product}
\mathbf{W}_{[(m,n), (i,j)]} = \mathcal{T}_{[(m,n), (i,j)]} = \int_{[1,a_1]\times [1,a_2]} \sqrt{\frac{1 + m(x^2 - 1) + n(y^2 - 1)}{1 + i(x^2 - 1) + j(y^2 - 1)}} E(dx, dy)
\end{eqnarray}
 Again using \eqref{Toral $k$-isometry}  and \eqref{Integral representation of moments product},  we have that 
\[
\sum_{\alpha \in \mathbb{Z}_+^2}
(-1)^{|\alpha|}
\binom{\beta}{\alpha}
\left\|
\mathbf{W}_{[{(s+\alpha_1,\; t+\alpha_2),(s,t)}]}f
\right\|^2 =0
\]
where $|\beta| = 2$, for all $s \ge 0$, $t \ge 0$ and for all $f\in \mathcal{H}$, we verified that $\mathbf{W}=\mathbf{T}$ is a toral $2$-isometry.

The following proposition shows that the joint kernel condition is preserved under taking the Cauchy dual of a pair of left-invertible commuting operators. 
\begin{proposition}\label{Cauchy dual}
  Let $\mathbf{T}=(T_1,T_2)$  be pair of left-invertible commuting operators on $B(\mathcal{H})$. Then the following are equivalent
  \begin{enumerate}
      \item[(i)] $\mathbf{T}$ satisfies the joint kernel condition,
      \item[(ii)] $T_i^*T_i(\ker \mathbf{T^*})=(ker \mathbf{T^*})$ $\text{ for i=1,2}$,
     \item[(iii)] The Cauchy dual $\mathbf{T'}=(T_1',T_2')$ satisfies the joint kernel condition. 
  \end{enumerate}
\begin{proof}
 The equivalence  $(i)\Longleftrightarrow (ii)$  follows from the fact that if $A\in B(\mathcal{H})$ is self-adjoint which is invertible in $B(\mathcal{H})$ and $\mathcal{L}$ is closed vector subspace of $\mathcal{H}$ which is invariant for $A$ then $A(\mathcal{L})=\mathcal{L}$ .
 
 The equivalence  $(ii)\Longleftrightarrow (iii)$  follows from simple observations $(T_i')'=T_i$ and $T_i'^*T_i'=(T_i^*T_i)^{-1}$ and $\ker T_i'^*=\ker T_i^*$ for $i=1, 2.$ 
\end{proof}  
\end{proposition}
\begin{remark}\label{Moment representation function}
For any $b,c\geq 0$ and nonnegative integers $m,n$, we have
\[
\frac{1}{1+mb+nc}
=\int_0^1 u^{mb+nc}\,du
=\int_0^1 (u^b)^m(u^c)^n\,du.
\]
\end{remark}

Now we are ready to prove Theorem \ref{Main theorem}.
\medskip
\begin{proof}[{Proof of Theorem \ref{Main theorem}:}]
Applying Proposition \ref{Cauchy dual}, we get 
$T_i'^* T_i' \big( \ker \mathbf{T}'^* \big) \subseteq \ker \mathbf{T}'^*$,  $\text{for } i = 1,2$. By Proposition \ref{Integral representation}, it suffices to consider $\mathbf{T}=\mathbf{W}=(T_1,T_2)$ is an operator-valued $2$-shifts on $\ell^2_{\mathbb{Z}_{+}^2}(\ker \mathbf{T}^*)$ with weights $\{W_{I}^{(j)}:j=1,2\}_{I\in \mathbb{Z}_{+}^2}$ which are given by in Proposition \ref{Integral representation}.
Since weights of $\mathbf{T}$ are invertible, self-adjoint and commuting, we infer from the Remark \ref{Cauchy dual weights} that $\mathbf{T'}$ is an operator-valued $2$-shifts on $\ell^2_{\mathbb{Z}_{+}^2}(\ker \mathbf{T}^*)$ with weights $\{(W_{I}^{(j)})^{-1}:j=1,2\}_{I\in \mathbb{Z}_{+}^2}$. Thus, by the commutativity of weights and the inversion formula for spectral integral, we have $(\mathbf{W}')_{[(m,n), (0,0)]} = \left( (\mathbf{W})_{[(m,n), (0,0)]} \right)^{-1}$
\[
\overset{\eqref{Integral representation of moments product}}{=} \int_{[1,a_1]\times [1,a_2]} \frac{1}{\sqrt{1 + m(x^2 - 1) + n(y^2 - 1)}} E(dx, dy).
\]
This in turn implies that \[
\| (\mathbf{W'})_{[(m,n), (0,0)]} f \|^2 = \int_{[1,a_1]\times [1,a_2]} \frac{1}{1 + m(x^2 - 1) + n(y^2 - 1)} \langle E(dx, dy)f, f \rangle.
\]
Now using Remark \ref{Moment representation function} for $b=x^2-1,$ $c=y^2-1,$ we have 
\[
\begin{aligned}
\left\|\mathbf W'_{[(m,n),(0,0)]}f\right\|^2
&=\int_{[1,a_1]\times[1,a_2]}
\int_0^1
(u^{x^2-1})^m(u^{y^2-1})^n
\,du\,dE_f(x,y).
\end{aligned}
\]
Now, consider the measurable mapping
\[
\Phi:[0,1]\times[1,a_1]\times[1,a_2]
\longrightarrow[0,1]^2
\]
defined by
\[
\Phi(u,x,y)=(u^{x^2-1},u^{y^2-1}).
\]
Let $\mu_f$ be the push-forward of the product measure
$du\otimes E_f$ under $\Phi$. That is,
\[
\mu_f(B)=(du\otimes E_f)(\Phi^{-1}(B))
\]
for every Borel set $B\subseteq[0,1]^2$.
By the integration formula for push-forward measures, we obtain
\[
\boxed{
\left\|\mathbf W'_{[(m,n),(0,0)]}f\right\|^2
=\int_{[0,1]^2}s^mt^n\,d\mu_f(s,t).
}
\]
Thus, the sequence of squared norms is represented by the
moments of a positive Borel measure on $[0,1]^2$.
Hence by Theorem \ref{Subnormality criteria}, $\mathbf{T}'=(T_1',T_2')$ is a subnormal operators pair. Since $T_1$ and $T_2$ are left-invertible, it follows from their polar decompositions that the Cauchy dual tuple $\mathbf{T}'=(T_1',T_2')$ is a tuple of contractions.
This completes the proof. 

\end{proof}

The following lemma plays an important role in ensuring that the Cauchy duals of the pairs in \cite[Theorem 3.4]{Mul}, namely $(U_1,M_z)$ and $(M_z,V_1)$, are commuting contractions and which is also helpful to prove Corollary~\ref{Seperate kernel condition} .
 \begin{lemma}\label{Commuting}
 Let $U$ be a unitary operator and let $T$ be an expansive operator on $B(\mathcal{H})$ such that $UT = TU$. Then the Cauchy dual tuple $(U', T')$ is commuting and consists of contractions. 
 \end{lemma}
\begin{proof}
   As $UT=TU$  therefore $UT^*=T^*U.$ Combining both of these, we have $U(T^*T)=(T^*T)U$, as $T^*T$ is invertible  so $U(T^*T)^{-1}=(T^*T)^{-1}U$ and hence $U'T'=T'U'$.
 \end{proof}
 We shall use the functional model and Wold-type decomposition for left-inverse commuting analytic toral $2$-isometries developed in \cite[Theorems 2.18, 3.3 and 3.4]{Mul}. In particular, the associated block matrix representation (see \cite[Theorem 3.4]{Mul}) will be used to prove the following corollary \ref{Seperate kernel condition}.

\begin{corollary}\label{Seperate kernel condition}
Let $\mathbf{T}=(T_1,T_2)$ be tuple of left-inverse commuting  toral $2$-isometry on $B(\mathcal{H})$ satisfying $T_i^*T_i(\ker T_i^*)\subseteq \ker T_i^*$ for $i=1,2.$ If $|T_1||T_2|=|T_2||T_1|$ on $\ker\mathbf{T}^*$ then the Cauchy dual tuple $\mathbf{T'}=(T_1',T_2')$ consists of subnormal contractions.
\end{corollary}
\begin{proof}
Using Theorem \cite[Theorem 3.4]{Mul} for $(T_1, T_2)$, we obtain a Wold-type decomposition. Therefore, by the given  kernel conditions, the operators $M_z'$ on $\mathcal{D}_{\mathcal{E}_{01}}(\mu_1)$ and $M_z'$ on $\mathcal{D}_{\mathcal{E}_{10}}(\mu_2)$ are subnormal contractions (see \cite[Theorem 3.3]{A.A}). 
Hence, by \cite[Proposition 2]{AV} and  Lemma \ref{Commuting}, the pairs $(U_1', M_z')$ and $(M_z', V_1')$ are subnormal. Moreover, using the separate kernel conditions $T_i^*T_i(\ker T_i^*)\subseteq \ker T_i^*$ for $i=1,2,$ we can deduce that \[
T_i^* T_i (\ker \mathbf{T}^*) \subseteq ( \ker \mathbf{T}^*), \quad \text{for } i=1,2 .
\] 
Therefore, by Theorem \ref{Main theorem} , it follows that $\mathbf{M}_{z}'=(M_{z_1}', M_{z_2}')$ is a subnormal contraction on $\mathcal{D}_{\ker \mathbf{T}^*}(\nu_1, \nu_2)$. 
Consequently, the tuple $\mathbf{T}'=(T_1', T_2')$ consists of  subnormal contraction. This completes the proof.
\end{proof}

 The following theorem enables us to construct an example of a toral $2$-isometric weighted shift whose initial two weights are positive and non commuting, while its Cauchy dual tuple $\mathbf{T'}=(T_1',T_2')$ is not subnormal (see Example \ref{Example 3}).

\begin{theorem}\cite[Theorem 4.1]{Athavale1987}\label{Example2}
Let $(S_1, S_2)$ be a commuting $2$-tuple of operators on $B(\mathcal{H})$.
Then the following are equivalent:
\begin{enumerate}
\item $(S_1, S_2)$ is a subnormal contraction.
\item for all non-negative integers $k_1, k_2$,
\[
\sum_{\substack{0 \le p_i \le k_i \\ i=1,2}} 
(-1)^{p_1 + p_2}
\binom{k_1}{p_1}
\binom{k_2}{p_2}
S_1^{*p_1} S_2^{*p_2} S_1^{p_1} S_2^{p_2}
\;\ge\; 0.
\]
\end{enumerate}
\end{theorem}
\begin{example}\label{Example 3}
In the setting of Theorem \ref{Example2}, $(T_1, T_2)$ is a toral $2$-isometric weighted shift with weights (See Theorem \ref{Representation of weights})
\[
W^{(1)}_{m,n}
=
\big[(m+1)A + nB + I\big]^{1/2}
\big[mA + nB + I\big]^{-1/2},
\]
\[
W^{(2)}_{m,n}
=
\big[mA + (n+1)B + I\big]^{1/2}
\big[mA + nB + I\big]^{-1/2}. 
\]
The Cauchy dual tuple $\mathbf{T'}=(T_1',T_2')$ is contractive weighted shift and has weights $\{(W_{I}^{(j)\,*})^{-1}: j=1,2\}_{I\in \mathbb{Z}_{+}^2}.$
Suppose that the Cauchy dual tuple  $\mathbf{T'}=(T_1',T_2')$ is subnormal. Taking $k_1 = k_2 = 1$ in Theorem \ref{Example2}, we obtain
\begin{eqnarray}\label{Positivity}
I - T_1'^* T_1' - T_2'^* T_2' + T_1'^* T_2'^* T_1' T_2' \geq 0,
\end{eqnarray}
We can rewrite \eqref{Positivity} as follows
\begin{eqnarray}\label{Norm positivity}
\|x\|^2 - \|T_1' x\|^2 - \|T_2' x\|^2 + \|T_1' T_2' x\|^2 \geq 0, \quad x \in \ell_{\mathbb{Z}_{+}^2}^2(\mathcal{H}).
\end{eqnarray}
Plugging $x = e_{0,0} \otimes f$ in \eqref{Norm positivity}, where $f \in \mathcal{H}$.
Then we get,
\[\|f\|^2
- \|(W_{0,0}^{(1)})^{*-1} f\|^2
- \|(W_{0,0}^{(2)})^{*-1} f\|^2
+ \|(W_{0,1}^{(1)})^{*-1} (W_{0,0}^{(2)})^{*-1} f\|^2
\geq 0.
\]
Since the  initial two weights are positive, the above inequality reduces to
\[
\|f\|^2
- \|(W_{0,0}^{(1)})^{-1} f\|^2
- \|(W_{0,0}^{(2)})^{-1} f\|^2
+ \|(W_{0,1}^{(1)})^{*-1} (W_{0,0}^{(2)})^{-1} f\|^2
\geq 0.
\]
Using Theorem \ref{Representation of weights} for $m=0,n=1,$ we have $(W_{0,1}^{(1)})^{*-1}
=
\big(W_{0,0}^{(1)2} + W_{0,0}^{(2)2} - I\big)^{-1/2}
\, W_{0,0}^{(2)}.$ Hence we obtain
\[
\|f\|^2
- \|(W_{0,0}^{(1)})^{-1} f\|^2
- \|(W_{0,0}^{(2)})^{-1} f\|^2
+ \big\|\big(W_{0,0}^{(1)2} + W_{0,0}^{(2)2} - I\big)^{-1/2} f\big\|^2
\geq 0.
\]
Equivalently, for all $f \in \mathcal{H}$,
\begin{eqnarray}\label{Final positivity}
\langle f,f\rangle
- \big\langle (W_{0,0}^{(1)})^{-2} f, f \big\rangle
- \big\langle (W_{0,0}^{(2)})^{-2} f, f \big\rangle 
+ \big\langle \big((W_{0,0}^{(1)})^2
+ (W_{0,0}^{(2)})^2 - I\big)^{-1} f, f \big\rangle
\geq 0.
\end{eqnarray}
We now show that \eqref{Final positivity} fails for a specific choice of weights.
 Let us take
\[
f = \begin{pmatrix} 1 \\ 0 \end{pmatrix}, \quad
W_{0,0}^{(1)} = \begin{pmatrix} 1 & 0 \\ 0 & 2 \end{pmatrix}, \quad
W_{0,0}^{(2)} = \begin{pmatrix} 2 & 1 \\ 1 & 2 \end{pmatrix}.
\]
Note that $W_{0,0}^{(1)}$ and $W_{0,0}^{(2)}$ are positive definite and do not commute and using \eqref{Final positivity}, we have a direct computation yields
\[
1 - 1 - \frac{5}{9} + \frac{8}{24}
= -\frac{2}{9} < 0.
\]
Therefore $\mathbf{T}'=(T_1', T_2')$ is not subnormal.
\end{example}
\section{\textbf{Characterizations of completely non-unitary toral $2$-isometries}}\label{Characterizations of Pure Toral $2$-Isometries}
An operator $T \in B(\mathcal{H})$ is said to be {\it completely non-unitary} (respectively, {\it pure}) if there exists no nonzero reducing closed subspace $\mathcal{L}$ of $\mathcal{H}$ for which the restriction $T|_{\mathcal{L}}$ is unitary (respectively, normal).
The unitary equivalence problem for pure $2$-isometries satisfying the kernel condition was completely solved in \cite[Theorem~3.3]{A.C}. We now establish the corresponding characterization for toral $2$-isometric operator-valued two-variable weighted shifts.
A commuting pair $(T_1,T_2)$ on $B(\mathcal{H})$ is called \textit{completely non-unitary} if both $T_1$ and $T_2$ are separately completely non-unitary.

The following model theorem, which is a part of Theorem \ref{Representation of weights}, classifies (up
to unitary equivalence) the set of analytic left-inverse commuting toral $2$-isometries satisfying the joint kernel condition.
\begin{theorem}\label{Characterization }
Let $\mathbf{T}=(T_1,T_2)$ be a left-inverse commuting tuple of bounded operators on a Hilbert space
$\mathcal H$, where $\mathcal H\neq\{0\}$. Then the following assertions are
equivalent.

\begin{enumerate}
\item[(i)] $\mathbf{T}$ is an analytic  toral $2$-isometry
satisfying the joint kernel condition.

\item[(ii)] $\mathbf{T}$ is a completely non-unitary commuting toral $2$-isometry
satisfying the joint kernel condition.


\item[(iii)] $\mathbf{T}$ is unitarily equivalent to an operator-valued two-variable
weighted shift with invertible weights $\{W_{I}^{(j)}:j=1,2\}_{I\in \mathbb{Z}_{+}^2},$ where the weights are given by
\[
W^{(1)}_{m,n}
=
\sqrt{(m+1)A+nB+I}\,
(\sqrt{mA+nB+I})^{-1},
\]
and
\[
W^{(2)}_{m,n}
=
\sqrt{mA+(n+1)B+I}\,
(\sqrt{mA+nB+I})^{-1},
\]
where $A=(W^{(1)}_{0,0})^2-I,$ $B=(W^{(2)}_{0,0})^2-I.$
\end{enumerate}
\end{theorem}
Moreover, if the initial weights  $W_{0,0}^{(1)},\;W_{0,0}^{(2)}$ commute, then the above Theorem \ref{Characterization } yields the following model theorem, which classifies (up to unitary equivalence) analytic left-inverse commuting toral $2$-isometries satisfying the joint kernel condition.
\begin{theorem}
Let $\mathbf{T}=(T_1,T_2)$ be a left-inverse commuting pair of bounded operators on a Hilbert space
$\mathcal H$, where $\mathcal H\neq\{0\}$. Assume that $|T_1||T_2| = |T_2||T_1|$ on $\ker \mathbf{T}^*.$
 Then the following assertions are equivalent.

\begin{enumerate}
\item[(i)] $\mathbf{T}$ is an analytic  toral $2$-isometry
satisfying the joint kernel condition.

\item[(ii)] $\mathbf{T}$ is a completely non-unitary commuting toral $2$-isometry
satisfying the joint kernel condition.


\item[(iii)] $\mathbf{T}$ is unitarily equivalent to an operator-valued two-variable
weighted shift with invertible weights $\{W_{I}^{(j)}:j=1,2\}_{I\in \mathbb{Z}_{+}^2},$  whose weights admit the integral representation,
\[
W^{(1)}_{m,n}
=
\int_{[1,a_1]\times[1,a_2]}
\sqrt{
\frac{1+(m+1)(x^2-1)+n(y^2-1)}
     {1+m(x^2-1)+n(y^2-1)}
}\,
dE(x,y),
\]
and
\[
W^{(2)}_{m,n}
=
\int_{[1,a_1]\times[1,a_2]}
\sqrt{
\frac{1+m(x^2-1)+(n+1)(y^2-1)}
     {1+m(x^2-1)+n(y^2-1)}
}\,
dE(x,y),
\]
where $E$ is a compactly supported
$B(\mathcal M)$-valued joint spectral measure on
$[1,a_1]\times[1,a_2]$ and $W_{0,0}^{(1)} = |T_1|$, $W_{0,0}^{(2)} = |T_2|$ on $\ker \mathbf{T}^*.$
\end{enumerate}
\end{theorem}
\begin{proof}
    The proof is part of Proposition \ref{Integral representation}.
\end{proof}
Now we are in position to answer the question when two analytic left-inverse commuting toral $2$-isometries satisfying the joint kernel condition are unitarily equivalent.
\begin{theorem}\label{thm:unitary-equivalence}
Let $\mathbf{T}=(T_1,T_2)$$\in \ell^2_{\mathrm{IW}}(\mathcal H)$ and
$\widetilde{\mathbf{T}}=(\widetilde{T}_1,\widetilde{T}_2)
\in \ell^2_{\mathrm{IW}}(\widetilde{\mathcal H})$
be commuting  operator-valued two-variable weighted shifts with weights $\{W_{I}^{(j)}:j=1,2\}_{I\in \mathbb{Z}_{+}^2}$ and $\{\widetilde W_{I}^{(j)}:j=1,2\}_{I\in \mathbb{Z}_{+}^2}$ respectively. Suppose the initial weights  $W_{0,0}^{(1)},\;W_{0,0}^{(2)},$ $\widetilde W_{0,0}^{(1)},$ and $\widetilde W_{0,0}^{(2)}$ are positive.  
 Then the following are equivalent:
\begin{enumerate}
\item[(i)] $\mathbf{T}$ and $\widetilde{\mathbf{T}}$ are unitarily equivalent.

\item[(ii)] There exists a unitary operator $U_{00}:\mathcal H\longrightarrow\widetilde{\mathcal H}$
such that $|\widetilde P_{m,n}|
=
U_{0,0}|P_{m,n}|U_{0,0}^{*},$ where $\mathcal{P}_{m,n}$ is as defined in \eqref{P-mn} for each $m,n\in\mathbb Z_+$.
\end{enumerate}
\end{theorem}

\begin{proof}

$(i)\Rightarrow(ii):$
Assume that $\mathbf{T}$ and $\widetilde{\mathbf{T}}$ are unitarily equivalent. Then there exists a unitary operator
$U:\ell^2(\mathbb Z_{+}^2)\otimes\mathcal H
\longrightarrow
\ell^2(\mathbb Z_{+}^2)\otimes\widetilde{\mathcal H}$
such that $UT_i=\widetilde T_iU$ for $i=1,2.$
Since both weighted shifts are analytic,
$\ker\mathbf{T}^*=e_{0,0}\otimes\mathcal H,$ and
$\ker\widetilde{\mathbf{T}}^*
=
e_{0,0}\otimes\widetilde{\mathcal H}.$
Taking adjoints, we have $T_i^*U^*=U^*\widetilde T_i^*,$
and hence $U(\ker\mathbf{T}^*)=\ker\widetilde{\mathbf{T}}^*.$
For $(m,n)\in\mathbb Z_+^2$, denote
$e_{m,n}\otimes\mathcal H
:=
\{\,e_{m,n}\otimes h:\;h\in\mathcal H\,\}.$
Since $U(e_{0,0}\otimes\mathcal H)=e_{0,0}\otimes\widetilde{\mathcal H},$
there exists a unique unitary operator
$U_{0,0}:\mathcal H\longrightarrow\widetilde{\mathcal H}$
such that $U(e_{0,0}\otimes x)=e_{0,0}\otimes U_{0,0}x,$ for $x\in\mathcal H.$ For each $(m,n)\in\mathbb Z_+^2$, let
$\mathcal M_{m,n}=T_1^mT_2^n(\ker\mathbf{T}^*),$
and $\widetilde{\mathcal M}_{m,n}=\widetilde T_1^m\widetilde T_2^n(\ker\widetilde{\mathbf{T}}^*).$
Since $U(\mathcal M_{m,n})=\widetilde{\mathcal M}_{m,n},$
there exists a unique unitary operator
$U_{m,n}:\mathcal H\longrightarrow\widetilde{\mathcal H}$ such that $U(e_{m,n}\otimes x)=e_{m,n}\otimes U_{m,n}x,$ for $x\in\mathcal H.$  Using the intertwining relation, $UT_i=\widetilde T_iU,$ for $i=1,2,$ we obtain
\begin{eqnarray}\label{First direction unitary weights}
U_{m+1,n}W_{m,n}^{(1)}
=
\widetilde W_{m,n}^{(1)}U_{m,n}.
\end{eqnarray}
\begin{eqnarray}\label{second direction unitary weights}
U_{m,n+1}W_{m,n}^{(2)}
=
\widetilde W_{m,n}^{(2)}U_{m,n}.
\end{eqnarray}
Let $P_{m,n}$ and $\widetilde P_{m,n}$ denote the corresponding moment operators (see Definition \ref{Moment Product}). Repeated application of the identities \eqref{First direction unitary weights} and \eqref{second direction unitary weights} yields $\widetilde P_{m,n}
=
U_{m,n}P_{m,n}U_{0,0}^{*}.$ Therefore,
\[
\begin{aligned}
|\widetilde P_{m,n}|^2
&=
\widetilde P_{m,n}^{*}\widetilde P_{m,n}\\
&=
U_{0,0}P_{m,n}^{*}P_{m,n}U_{0,0}^{*}\\
&=
U_{0,0}|P_{m,n}|^2U_{0,0}^{*}.
\end{aligned}
\]
Since both $|P_{m,n}|$ and $|\widetilde P_{m,n}|$ are positive, uniqueness of positive square roots gives
\begin{eqnarray}\label{Unitary of Moment product}
|\widetilde P_{m,n}|
=
U_{0,0}|P_{m,n}|U_{0,0}^{*}.
\end{eqnarray}
Hence (ii) holds.

$(ii)\Rightarrow(i)$:
Suppose there exists a unitary operator $U_{0,0}:\mathcal H\longrightarrow\widetilde{\mathcal H}$ such that $|P_{m,n}|=U_{0,0}^*|\widetilde P_{m,n}|U_{0,0},$ for all $(m,n)\in\mathbb Z_+^2.$ Therefore for any $(m,n)\in\mathbb Z_+^2$ and $h\in\mathcal H,$ 
\[
\|P_{m,n}h\|
=\||P_{m,n}|h\|
=\||\widetilde P_{m,n}|U_{0,0}h\|
=\|\widetilde P_{m,n}U_{0,0}h\|.
\]
For each $(m,n)\in\mathbb Z_+^2,$ since $P_{m,n}$ has dense range, there exists a unique unitary operator $U_{m,n}:\mathcal{H}\longrightarrow \widetilde {\mathcal{H}}$ such that $U_{m,n}P_{m,n}=\widetilde P_{m,n}U_{0,0}.$
Define $U:\ell^2(\mathbb Z_{+}^2)\otimes\mathcal H
\longrightarrow
\ell^2(\mathbb Z_{+}^2)\otimes\widetilde{\mathcal H}$ as
\[
U(e_{m,n}\otimes h)=e_{m,n}\otimes U_{m,n}h,
\qquad h\in\mathcal H,\ (m,n)\in\mathbb Z_+^2.
\]
Since each $U_{m,n}$ is unitary, $U$ is unitary.

Let $x\in\mathcal H$ and $m,n\in\mathbb Z_+$. Since $P_{m,n}$ is invertible, there exists $h\in\mathcal H$ such that
$x=P_{m,n}h$. Then
\[
\begin{aligned}
UT_1(e_{m,n}\otimes x)
&=UT_1(e_{m,n}\otimes P_{m,n}h)\\
&=U(e_{m+1,n}\otimes P_{m+1,n}h)\\
&=e_{m+1,n}\otimes U_{m+1,n}P_{m+1,n}h\\
&=e_{m+1,n}\otimes\widetilde P_{m+1,n}U_{0,0}h.
\end{aligned}
\]
On the other hand,
\[
\begin{aligned}
\widetilde T_1U(e_{m,n}\otimes x)
&=\widetilde T_1(e_{m,n}\otimes U_{m,n}P_{m,n}h)\\
&=\widetilde T_1(e_{m,n}\otimes\widetilde P_{m,n}U_{0,0}h)\\
&=e_{m+1,n}\otimes\widetilde P_{m+1,n}U_{0,0}h.
\end{aligned}
\]
Hence, $UT_1(e_{m,n}\otimes x)
=
\widetilde T_1U(e_{m,n}\otimes x).$ 
Since the space $\mbox{span}\{e_{m,n}\otimes x:\,x\in\mathcal H,\ (m,n)\in\mathbb Z_+^2\}$ is dense in 
$\ell^2(\mathbb Z_{+}^2)\otimes\mathcal H$, we conclude that $UT_1=\widetilde T_1U.$
Similarly, $UT_2=\widetilde T_2U.$
Therefore, $\mathbf T\cong\widetilde{\mathbf T}.$
This completes the proof.
\end{proof}
The following theorem provides a complete characterization of unitary
equivalence for toral $2$-isometric operator-valued two-variable weighted shifts whose initial weights are commuting positive operators. It shows that the entire weighted shift is determined, up to unitary equivalence, by its initial weights, or equivalently, by the associated joint spectral measure.
\begin{theorem}\label{thm:UE-commuting} 
Let $\mathbf{T}=(T_1,T_2)$$\in \ell^2_{\mathrm{IWP}}(\mathcal H)$ and
$\widetilde{\mathbf{T}}=(\widetilde{T}_1,\widetilde{T}_2)
\in \ell^2_{\mathrm{IWP}}(\widetilde{\mathcal H})$
be toral $2$-isometric operator-valued two-variable weighted shifts with weights $\{W_{I}^{(j)}:j=1,2\}_{I\in \mathbb{Z}_{+}^2}$ and $\{\widetilde W_{I}^{(j)}:j=1,2\}_{I\in \mathbb{Z}_{+}^2}$ respectively. Further, suppose that the pairs $(W_{0,0}^{(1)},W_{0,0}^{(2)})$ and $(\widetilde W_{0,0}^{(1)},\widetilde W_{0,0}^{(2)})$ are commuting pairs of  operators. 
Then the following are equivalent.
\begin{enumerate}
\item[(i)]
$\mathbf{T}$ and $\widetilde {\mathbf{T}}$ are unitarily equivalent.

\item[(ii)]
There exists a unitary operator $U_0:\mathcal H\longrightarrow\widetilde{\mathcal H}$ such that
\[
U_0W_{0,0}^{(i)}
=
\widetilde W_{0,0}^{(i)}U_0,
\qquad i=1,2.
\]
\item[(iii)] 
The joint spectral measures $E$ and $\widetilde E$ are unitarily equivalent.
\end{enumerate}
\end{theorem}
\begin{proof}
$(i)\Rightarrow(ii):$ The proof of this part 
 is identical to that of Theorem~\ref{thm:unitary-equivalence}. Taking  $(m,n)$ to be $(1,0)$ and $(0,1)$ in \eqref{Unitary of Moment product}, we obtain
\[
|\widetilde W_{0,0}^{(1)}|
=
U_{0,0}|W_{0,0}^{(1)}|U_{0,0}^{*},
\]
and
\[
|\widetilde W_{0,0}^{(2)}|
=
U_{0,0}|W_{0,0}^{(2)}|U_{0,0}^{*}.
\]
Since the initial weights are positive, therefore $|W_{0,0}^{(i)}|=W_{0,0}^{(i)},$ and $|\widetilde W_{0,0}^{(i)}|=\widetilde W_{0,0}^{(i)},$
we conclude that $U_{0,0}W_{0,0}^{(i)}=\widetilde W_{0,0}^{(i)}U_{0,0},$ for $i=1,2.$

$(ii)\Rightarrow (i):$
Assume that there exists a unitary operator $U_0:\mathcal H\to\widetilde{\mathcal H}$ such that $U_0W_{0,0}^{(i)}
=
\widetilde W_{0,0}^{(i)}U_0,$ for $i=1,2.$
Since initial weights have integral representation with respect to the spectral measure so all the weights are commuting and therefore by the \cite[Theorem 5.4.9]{BM}, $U_0W_{m,n}^{(i)}=\widetilde W_{m,n}^{(i)}U_0,$ for all $i=1,2,\mbox{ and } (m,n)\in\mathbb Z_+^2.$ 
Define $U:\ell^2(\mathbb Z_{+}^2)\otimes\mathcal H
\longrightarrow
\ell^2(\mathbb Z_{+}^2)\otimes\widetilde{\mathcal H},$ a linear operator, by the rule $U(e_{m,n}\otimes h)=e_{m,n}\otimes U_0h$ for $m,n\in\mathbb Z_+$ and $h\in\mathcal H.$
Since $U_0$ is unitary, $U$ is unitary. Moreover, for $i=1,2,$
\[
UT_i(e_{m,n}\otimes h)
=
e_{m+1,n}\otimes U_0W_{m,n}^{(i)}h
=
e_{m+1,n}\otimes
\widetilde W_{m,n}^{(i)}U_0h
=
\widetilde T_iU(e_{m,n}\otimes h).
\]
Therefore, $UT_i=\widetilde T_iU,$ for $i=1,2,$
and hence $\mathbf{T}\cong\widetilde {\mathbf{T}}.$

$(ii)\Rightarrow(iii):$
 Suppose $(ii)$ holds.
By Proposition \ref{Integral representation}, the weights sequence $\{W_{I}^{(j)}:j=1,2\}_{I\in \mathbb{Z}_{+}^2}$ are given by the integral
representations \eqref{Integral formula for horizontal weights} and \eqref{Integral formula for vertical weights}. In particular, 
\[
W_{0,0}^{(1)}
=
\int_{[1,a_1]\times[1,a_2]}x\,dE(x,y),
\]
and
\[
W_{0,0}^{(2)}
=
\int_{[1,a_1]\times[1,a_2]}y\,dE(x,y),
\]
where $E$ is the joint spectral measure associated with
$\big(W_{0,0}^{(1)},W_{0,0}^{(2)}\big)$.

Similarly,
\[
\widetilde W_{0,0}^{(1)}
=
\int_{[1,\widetilde a_1]\times[1,\widetilde a_2]}
x\,d\widetilde E(x,y),
\]
and
\[
\widetilde W_{0,0}^{(2)}
=
\int_{[1,\widetilde a_1]\times[1,\widetilde a_2]}
y\,d\widetilde E(x,y),
\]
where $\widetilde E$ is the joint spectral measure associated with
$\big(\widetilde W_{0,0}^{(1)},\widetilde W_{0,0}^{(2)}\big)$.
For all Borel subsets, $\Delta\subseteq [1,a_1]\times[1,a_2],$ define $F(\Delta)=U_0^{*}\widetilde E(\Delta)U_0.$ 
Then
\[
\begin{aligned}
\int x\,dF
&=
U_0^{*}
\left(
\int x\,d\widetilde E
\right)
U_0      \\
&=
U_0^{*}\widetilde W_{0,0}^{(1)}U_0\\
&=
W_{0,0}^{(1)}.
\end{aligned}
\]
Similarly,
\[
\begin{aligned}
\int y\,dF
&=
U_0^{*}
\left(
\int y\,d\widetilde E
\right)
U_0\\
&=
U_0^{*}\widetilde W_{0,0}^{(2)}U_0\\
&=
W_{0,0}^{(2)}.
\end{aligned}
\]
Since the joint spectral measure is uniquely determined by the commuting pair,
we conclude that $F=E$. Hence $U_0E(\Delta)=\widetilde E(\Delta)U_0,$ for all Borel subset, $\Delta\subseteq [1,a_1]\times[1,a_2].$

\medskip

$(iii)\Rightarrow(i):$
Suppose (iii) holds.
Therefore,
\[
\begin{aligned}
U_0W_{m,n}^{(1)}
&=
U_0
\int
\sqrt{
\frac{1+(m+1)(x^2-1)+n(y^2-1)}
{1+m(x^2-1)+n(y^2-1)}
}\,
dE(x,y)                                     \\
&=
\int
\sqrt{
\frac{1+(m+1)(x^2-1)+n(y^2-1)}
{1+m(x^2-1)+n(y^2-1)}
}\,
d\widetilde E(x,y)\,
U_0                                          \\
&=
\widetilde W_{m,n}^{(1)}U_0.
\end{aligned}
\]
Similarly
\[
\begin{aligned}
U_0W_{m,n}^{(2)}
&=
U_0
\int
\sqrt{
\frac{1+m(x^2-1)+(n+1)(y^2-1)}
{1+m(x^2-1)+n(y^2-1)}
}\,
dE(x,y)                                     \\
&=
\int
\sqrt{
\frac{1+m(x^2-1)+(n+1)(y^2-1)}
{1+m(x^2-1)+n(y^2-1)}
}\,
d\widetilde E(x,y)\,
U_0                                          \\
&=
\widetilde W_{m,n}^{(2)}U_0.
\end{aligned}
\]
Define $U:\ell^2(\mathbb Z_{+}^2)\otimes\mathcal H
\longrightarrow
\ell^2(\mathbb Z_{+}^2)\otimes\widetilde{\mathcal H}$
by $U(e_{m,n}\otimes h)=e_{m,n}\otimes U_0h,$ for $(m,n)\in\mathbb Z_+^2,\;h\in\mathcal H.$
Since $U_0$ is unitary, $U$ is unitary. Moreover,
\[
\begin{aligned}
UT_1(e_{m,n}\otimes h)
&=
U(e_{m+1,n}\otimes W_{m,n}^{(1)}h)\\
&=
e_{m+1,n}\otimes U_0W_{m,n}^{(1)}h\\
&=
e_{m+1,n}\otimes
\widetilde W_{m,n}^{(1)}U_0h\\
&=
\widetilde T_1U(e_{m,n}\otimes h).
\end{aligned}
\]
Similarly, $UT_2=\widetilde T_2U.$ Hence $UT_i=\widetilde T_iU$
for $i=1,2$
which proves that $\mathbf{T}\cong\widetilde {\mathbf{T}}.$ This completes the proof.
\end{proof}
\noindent\textbf{Acknowledgment:} I would like to sincerely thank my supervisor Dr. Rajeev Gupta for his invaluable guidance and constant support during the preparation of this article and also sincerely thanks Prof. Sameer Chavan for his several valuable inputs and insightful comments, which greatly helped in the preparation of this article.

\bibliographystyle{amsplain}

\begin{thebibliography}{99}

\bibitem{A.J}
J. Agler,
A disconjugacy theorem for Toeplitz operators,
\textit{Amer. J. Math.} {\bf 112} (1990), 1--14.
\bibitem{A.A}
A. Anand, S. Chavan, Z. J. Jabłoński, and J. Stochel,
A solution to the Cauchy dual subnormality problem for 2-isometries,
\textit{J. Funct. Anal.} {\bf 277} (2019), 108292, 51 pp.
\bibitem{A.C} A. Anand, S. Chavan, Z. J. Jabłoński, and J. Stochel, Complete systems of unitary invariants for some classes of 2-isometries,
\textit{Banach J. Math. Anal.} {\bf 13} (2), 2019, 359-385

\bibitem{A.K}
A. Anand, S. Chavan, and R. Nailwal,
Joint complete monotonicity of rational functions in two variables and toral m-isometric pairs,
\textit{J. Operator Theory} Volume {\bf 92} (2024), 101--130.

\bibitem{Athavale1987}
A. Athavale,
Holomorphic kernels and commuting operators,
\textit{Trans. Amer. Math. Soc.} {\bf 304} (1987), 101--110.
\bibitem{AV}
A. Athavale and S. Pedersen,
Moment problems and subnormality,
\textit{J. Math. Anal. Appl.} {\bf 146} (1990), 434--441.
\bibitem{V.A}
A. Athavale and V. M. Sholapurkar,
Completely hyperexpansive operator tuples,
\textit{Positivity} {\bf 3} (1999), 245--257.
\bibitem{AM} A. Aleman and B. Malman, Hilbert spaces of analytic functions with a contractive
backward shift, \textit{J. Funct. Anal.} {\bf 277} (2019), 157-199.
\bibitem{BM}M. Sh. Birman and M. Z. Solomjak, \textit{Spectral theory of selfadjoint operators in Hilbert space,}
D. Reidel Publishing Co., Dordrecht, 1987
\bibitem{B.C} C. Badea, L. Suciu, The Cauchy dual and 2-isometric liftings of concave operators,
\textit{J. Math. Anal. Appl.} {\bf 472} (2019), 1458-1474.
\bibitem{S.B}
S. Bera, S. Chavan, and S. Ghara,
Dirichlet-type spaces of the unit bidisc and toral 2-isometries,
\textit{Canad. J. Math.} \textbf{77} (4) (2025), 1271--1293.
\bibitem{Mul}
M. Bhattacharjee, R. Gupta, and V. Venugopal,
Multivariable Wold-type decomposition and analytic models for a class of left-inverse commuting pairs,
\textit{J. Math. Anal. Appl.} \textbf{565} (1) (2027), Article 130965.
\bibitem{MJ}
M. Bhattacharjee, R. Gupta, and V. Venugopal,
Dirichlet type spaces in the unit bidisc and Wandering Subspace Property for operator tuples,
\textit{Journal of Geometric Analysis}, to appear.

 \bibitem{C.B}C. Berg, J. P. R. Christensen, and P. Ressel,
\textit{Harmonic Analysis on Semigroups,}
Springer-Verlag, Berlin, 1984.
 \bibitem{S.C}
S. Chavan,
On operators Cauchy dual to 2-hyperexpansive operators,
\textit{Proc. Edinb. Math. Soc.} {\bf 50} (2007), 637--652.
\bibitem{CJ}S. Chavan, Z. J. Jabłoński, I. B. Jung and J. Stochel, Taylor spectrum approach to
Brownian-type operators with quasinormal entry, \textit{Annali di Matematica Pura ed Applicata,} {\bf 200} (2021), no. 3, 881–922.
\bibitem{CH}S. Chavan, S. Ghara, Md. R. Reza, The Cauchy dual subnormality problem via de
Branges-Rovnyak spaces, \textit{Studia Math} {\bf 265} (2022), 315-341.
\bibitem{R.E}R. E. Curto, M. Putinar, Existence of non-subnormal polynomially hyponormal operators,
\textit{Bull. Amer. Math. Soc.} {\bf 25} (1991), 373-378.
\bibitem{R.P}R. E. Curto, M. Putinar, Nearly subnormal operators and moment problems, \textit{J. Funct. Anal.}
{\bf 115} (1993), 480-497.

\bibitem{R.S}
R. Gupta, S. Kumar, and S. Trivedi,
Von Neumann's inequality for commuting operator-valued multishifts,
\textit{Proc. Amer. Math. Soc.} {\bf 147} (2019), 2599--2608.
\bibitem{R.G}
R. Gupta, S. Kumar, and S. Trivedi,
Unitary equivalence of operator-valued multishifts,
\textit{J. Math. Anal. Appl.} {\bf 487} (2020), 23 pp.

\bibitem{R.J}
R. Gupta and G. Misra,
The Carathéodory--Fejér interpolation on the polydisc,
\textit{Studia Math.} {\bf 254} (2020), no. 3, 265--294.

\bibitem{HA}
P. R. Halmos,
\textit{A Hilbert Space Problem Book,}
Springer-Verlag, New York, 1982.

\bibitem{Z.J}
Z. Jabłoński,
Complete hyperexpansivity, subnormality and inverted boundedness conditions,
\textit{Integral Equations Operator Theory} {\bf 44} (2002), 316--336.
\bibitem{J.Z}
Z. Jabłoński,
Hyperexpansive operator-valued unilateral weighted shifts,
\textit{Glasg. Math. J.} {\bf 46} (2004), 405--416.

\bibitem{LS}
S. Luo, C. Gu and S. Richter,
Higher order local Dirichlet integrals and de Branges--Rovnyak spaces,
\textit{Adv. Math.} \textbf{385} (2021), Paper No. 107748.

\bibitem{M.M}
M. Martin and M. Putinar,
\textit{Lectures on Hyponormal Operators},
Operator Theory: Advances and Applications, Vol. \textbf{39},
Birkhäuser, Basel, 1989.
\bibitem{Nailwal2024}
R.~Nailwal, A complete solution to the Cauchy dual subnormality problem
for torally expansive toral $3$-isometric weighted $2$-shifts,
\textit{New York J. Math.} {\bf 30} (2024), 1517--1533.
\bibitem{S.R}
S. Richter,
Invariant subspaces of the Dirichlet shift,
\textit{J. Reine Angew. Math.} {\bf 386} (1988), 205--220.

\bibitem{Simon2015}
B. Simon, \textit{Operator Theory: A Comprehensive Course in Analysis, Part 4},
American Mathematical Society, Providence, RI, 2015.

\bibitem{S.S}
S. Shimorin,
Wold-type decompositions and wandering subspaces for operators close to isometries,
\textit{J. Reine Angew. Math.} {\bf 531} (2001), 147--189.

\bibitem{SH} S. Shimorin, Complete Nevanlinna–Pick property of Dirichlet-type spaces, \textit{J. Funct.
Anal.} {\bf 191} (2002), 276-296.
\bibitem{B.S}
B. Simon, The classical moment problem as a self-adjoint finite difference operator,
\textit{Adv. Math.} \textbf{137} (1998), 82--203.
\bibitem{St}
J. Stochel,
Characterizations of subnormal operators,
\textit{Studia Math.} {\bf 97} (1991), 227--238.
\bibitem{B.W}
B. Wróbel,
Joint spectral multipliers for mixed systems of operators,
\textit{J. Fourier Anal. Appl.} {\bf 23} (2017), no. 2, 245--287.

\end{thebibliography}

\end{document}